\documentclass{amsart}
\usepackage{amsfonts}
\usepackage{mathrsfs}
\usepackage{enumerate}

\RequirePackage{amsmath} \RequirePackage{amssymb}
\usepackage{amscd,latexsym,amsthm,amsfonts,amssymb,amsmath,amsxtra}
\usepackage[all]{xy}
\usepackage[mathscr]{eucal}
\usepackage[colorlinks=true, urlcolor=blue,bookmarks=true,bookmarksopen=true,
citecolor=blue,hypertex]{hyperref}

\newcommand{\ra}{\rightarrow}

    \newcommand{\BA}{{\mathbb {A}}} 
    \newcommand{\BC}{{\mathbb {C}}} 
     \newcommand{\BF}{{\mathbb {F}}}
    \newcommand{\BG}{{\mathbb {G}}}

     \newcommand{\BR}{{\mathbb {R}}}

    \newcommand{\BW}{{\mathbb {W}}} 
     \newcommand{\BZ}{{\mathbb {Z}}}
    \newcommand{\CA}{{\mathcal {A}}} 
    \newcommand{\CC}{{\mathcal {C}}} 
     
    \newcommand{\CG}{{\mathcal {G}}}

    \newcommand{\CS}{{\mathcal {S}}}

     \newcommand{\Aut}{{\mathrm{Aut}}}

    \renewcommand{\Re}{{\mathrm{Re}}} 
    \newcommand{\Res}{{\mathrm{Res}}}

     \newcommand{\tr}{{\mathrm{tr}}}
      
    \newcommand{\vol}{{\mathrm{vol}}}

    \newcommand{\wt}{\widetilde} 
    
    \newcommand{\pair}[1]{\langle {#1} \rangle}

    \newcommand{\bs}{\backslash}
    \theoremstyle{plain}
    \newtheorem{thm}{Theorem}[section] \newtheorem{cor}[thm]{Corollary}
    \newtheorem{lem}[thm]{Lemma}  \newtheorem{prop}[thm]{Proposition}
     
    \newcommand{\Her}{\mathrm{Her}}
    
    \newcommand{\diag}{\mathrm{diag}}
    \newcommand{\dif}{\mathrm{d}}

  \numberwithin{equation}{section}

\begin{document}

\title[Siegel-Weil formula over function fields]{On the Siegel-Weil formula for classical groups over function fields}
\author{Wei Xiong}
\address{School of Mathematics, Hunan University,
Changsha 410082, China}
\email{weixiong@amss.ac.cn}
\thanks{}


\date{\today}

\dedicatory{Dedicated to the memory of my mother}


\begin{abstract}
We establish a Siegel-Weil formula for classical groups over a function field with odd characteristic, which asserts in many cases that the Siegel Eisenstein series is equal to an integral of a theta function. This is a function-field analogue of the classical result proved by A. Weil in his 1965 Acta Math. paper. We also give a convergence criterion for the theta integral by using Harder's reduction theory over function fields.
\end{abstract}

\maketitle

\tableofcontents

\section*{Introduction}
Let $k$ be a global function field with $\mathrm{char}(k)\neq 2$, and let $D$ be a quadratic field extension of $k$. Let $G=U(n, n)$ be the quasi-split unitary group in $2n$ variables over $D$, and let $H=U(V)$ be the unitary group associated to a non-degenerate hermitian space $V$ over $D$ of dimension $m$ and of Witt index $r$. In this paper, we study the Siegel-Weil formula for the pair $(G, H)$. Let $\BA$ be the ring of adeles of $k$, and let $X=V^n\cong M_{m\times n}(D)$. With some extra data, there is Weil representation $\omega$ of $G(\BA)$ on the space $\CS(X(\BA))$ of Schwartz-Bruhat functions on $X(\BA)$.

For $\Phi \in \CS(X(\BA))$, define the theta integral
\[
I(\Phi)=\int_{H(\BA)/H(k)} \sum_{x\in X(k)}\Phi(hx)~\dif h,
\]
where $\dif h$ is the Haar measure on $H(\BA)$ such that $\vol(H(\BA)/H(k))=1$;
and define the Siegel Eisenstein series
\[
E(\Phi)=\sum_{\gamma \in P(k)\bs G(k)}\omega(\gamma)\Phi(0),
\]
where $P$ is the Siegel parabolic subgroup of $G$.

In this paper, we will establish a convergence criterion for the theta integral and prove a Siegel-Weil formula. More precisely, we will prove that the theta integral $I(\Phi)$ is absolutely convergent for any $\Phi \in \CS(X(\BA))$ whenever $r=0$ or $m-r>n$, and that the Siegel Eisenstein series $E(\Phi)$ is absolutely convergent for any $\Phi \in \CS(X(\BA))$ whenever $m>2n$; moreover, we will prove the following {\sl Siegel-Weil formula}:
\[
\text{If $m>2n$, then $I(\Phi)=E(\Phi)$ for any $\Phi \in \CS(X(\BA))$.}
\]
These results are analogues of those in \cite{Weil 1965}. We will also consider the cases where $D$ is a division algebra over $k$ whose center is $k$ or a quadratic field extension of $k$, and we will obtain similar results.  We  follow the approach in \cite{Weil 1965} to prove these results; in particular, we will make use of the reduction theory over function fields established in \cite{Harder 1969} and the theory of Eisenstein series over function fields in \cite{Morris 1982}. For the convergence criterion for the theta integral, we here take Weil's approach in \cite{Weil 1965}. Note that for the  Siegel-Weil formula in this case, besides Weil's approach in \cite{Weil 1965}, other approaches are applicable, see for example \cite{Kudla-Rallis 1988b} and \cite{Urt}.

The modern study of the Siegel-Weil formulas over number fields begins with the fundamental works of Weil in 1960s, and after the seminal works of Kudla and Rallis in 1980s and 1990s, it has been well developed in recent years, with a major breakthrough made by Gan, Qiu, and Takeda in early 2010s. For this development, see for example \cite{Weil 1964}, \cite{Rallis 1987}, \cite{Kudla-Rallis 1988a}, \cite{Kudla-Rallis 1988b}, \cite{Kudla-Rallis 1994}, \cite{Ikeda 1996}, \cite{Moe97}, \cite{Tan 1998}, \cite{Ichino 2001}, \cite{Ichino 2004}, \cite{Ichino 2007}, \cite{Jiang-Soudry 2007}, \cite{Urt}, \cite{Gan-Takeda 2011}, \cite{Yamana 2011}, \cite{Xiong 2013}, \cite{Yamana 2013a}, \cite{Yamana 2013b}, and especially \cite{Gan-Qiu-Takeda 2014} and the references cited therein.

On the other hand, the Siegel-Weil formulas over function fields are only studied in some special cases to date. The first result on Siegel-Weil formulas over function fields is Haris' pioneering work in 1974 (\cite{Haris 1974}). Thereafter only little relevant research has been done. Recently, F.-T. Wei (\cite{Wei 2015}) proved a Siegel-Weil formula for anisotropic quadratic forms over function fields.

In this paper, we consider the situation similar to that in \cite{Weil 1965}, and we follow closely the strategy in \cite{Weil 1965} to prove the Siegel-Weil formula. In particular, many auxiliary results are proved following the methods in \cite{Weil 1965} with some necessary changes from number-field situation to function-field situation.

It is worth mentioning that Weil already had the function-field case in mind, and he had proved many preliminary results in \cite{Weil 1964} and \cite{Weil 1965}. But due to the lack of the reduction theory over function fields right then, he only considered the Siegel-Weil formula over number fields. Fortunately for our current purpose, the reduction theory for reductive algebraic groups over function fields has been established by Harder in late 1960s (see \cite{Harder 1969} or \cite{Springer 1994}). We emphasize that the reduction theory for reductive algebraic groups over function fields has played an essential role in establishing the convergence criteria for the theta integral and in the proof of the Siegel-Weil formula.

We give an overview of this paper in the following, specializing in the case where $D$ is a quadratic field extension of $k$, $G=U(n, n)$ is the isometry group associated to a split skew-hermitian space of dimension $2n$ over $D$, and $H=U(V)$ is the isometry group associated to a non-degenerate hermitian space $V$ over $D$ of dimension $m$ and of Witt index $r$. The hermitian form on $V$ is denoted by $(~,~): V \times V \ra D$. We allow $n=0$, which then means that $G=\{1\}$ is the trivial one-element group.

In Section \ref{Section 2}, we will give various notation and conventions used in this paper. We regard the function field $k$ as a finite extension of the rational function field $\BF_q(x)$ over the finite field $\BF_q$ with $q$ elements. We will introduce various algebraic groups over $k$. There is the Siegel parabolic subgroup $P=NM$ of $G$, where $N\cong \Her_n$ is the unipotent radical and $M\cong \Res_{D/k}GL_n$ is the Levi subgroup, here $\Her_n(k)$ is the space of hermitian matrices of order $n$ over $D$. Let $\BA$ be the ring of adeles of $k$. We have an Iwasawa decomposition $G(\BA)=P(\BA)G(O_{\BA})=N(\BA)M(\BA)G(O_{\BA})$, where $O_{\BA}=\prod_v O_v$ is the set of integral adeles with $O_v$ being the ring of integers of the local field $k_v$ for any place $v$ of $k$. Thus we can write any element $g\in G(\BA)$ in the form $g=n(b)m(a)g_1$ with $n(b)\in N(\BA)$, $m(a)\in M(\BA)$ with $a\in GL_n(D_{\BA})$ and $g_1\in G(O_{\BA})$, and we let $|a(g)|=|\det(a)|_{D_{\BA}}$. We let $X=V^n=\{(x_1, \ldots, x_n): x_i\in V\}$, and define a mapping $i_X: X \ra \Her_n$ by $i_X(x)=((x_i, x_j)_V)$ for $x=(x_1, \ldots, x_n)\in X$. For an integer $r$ with $0\leq r \leq n$, define a subspace $X_r$ of $X$ by $X_r=\{(x_1, \ldots, x_r, 0, \ldots, 0): x_i\in V\}$. Then $X_r\cong V^r$ and $X_n=X$. Let $\psi$ be a non-trivial character of $\BA/k$ and let $\chi$ be a Hecke character of $D$ satisfying $\chi|_{\BA^{\times}}=\epsilon_{D/k}^m$, where $\epsilon_{D/k}$ is the quadratic Hecke character of $k$ associated to the quadratic extension $D/k$; then there is an associated Weil representation $\omega=\omega_{\psi, \chi}$ of the adelic group $G(\BA)$ on the space $\CS(X(\BA))$ of Schwartz-Bruhat functions on $X(\BA)$.

In Section \ref{Section 3}, we will formulate some results in reduction theory over function fields  and prove some auxiliary lemmas (i.e. Lemmas \ref{Lem 4}, \ref{Lem 5}, \ref{Lem 6}). These lemmas are analogues of the results in \cite[n. 11--13]{Weil 1965}, and we will follow Weil's method to prove them. In particular, we will make use of two big theorems in reduction theory, i.e. the compactness theorem and the fundamental set theorem. The results in this section are essential for what follows and will allow us to follow Weil's strategy to prove the Siegel-Weil formula over function fields. We choose a place $v_0$ of the function field $k$ which lies above the place $(x^{-1})$ of the rational function field $\BF_q(x)$ and such that the residue field of $k$ at $v_0$ is $\BF_q$. The place $v_0$ is an analogue of the real place in the number field case and will play the similar role. We choose and fix a uniformizer $\varpi$ of $k_{v_0}$. Let $T\cong (\BG_m)^r$ be a maximal split torus in $H$, where $\BG_m=GL_1$ and  $r$ is the Witt index of $V$. We define a subgroup $\Theta(T)\cong \BZ^r$ of $T(\BA)$ by  $\Theta(T)=\{(a_{\tau_1}, \ldots, a_{\tau_r}): \tau_i \in \BZ\}$ via the isomorphism $(\BG_m)^r \cong T$, where $a_{\tau}$ is the idele whose component at $v_0$ is $\varpi^{\tau}$ and the other components are all equal to 1. Let $P_0$ be a minimal $k$-parabolic subgroup of $H$ which contains $T$, and let $P_0(\BA)^1$ be the subgroup of $P_0(\BA)$ formed of elements $p$ such that $|\lambda(p)|_{\BA}=1$ for any $k$-rational character $\lambda$ of $P_0$. Then $P_0(\BA)^1/P_0(k)$ is compact by the compactness theorem in reduction theory, and we have $P_0(\BA)/P_0(\BA)^1\cong \Theta(T)$. Using the fundamental set theorem in reduction theory over function fields, we will show in Lemma \ref{Lem 2.3} that there exists a compact subset $C$ of $H(\BA)$ such that
\[
H(\BA)=C\cdot \Theta^{+} \cdot P_0(\BA)^1 \cdot H(k),
\]
where
\[
\Theta^{+}=\{\theta \in \Theta(T): |\alpha(\theta)|_{\BA}\leq 1, \forall \alpha \in \Delta\},
\]
and $\Delta$ is the set of simple roots of $H$ relative to $T$. These results are analogues of some results in reduction theory over number fields as described in \cite{Weil 1965}. Next, we will follow Weil's approach to use these results to study the convergence of theta integrals. In particular, we will show in Lemma \ref{Lem 5} that the theta integral $I(\Phi)$ is absolutely convergent for any $\Phi \in \CS(X(\BA))$ whenever the integral
\[
\int_{\Theta^{+}}\prod_{\lambda} \sup(1, |\lambda(\theta)|_{\BA}^{-1})^{m_{\lambda}}\cdot |\Delta_{P_0}(\theta)|_{\BA}^{-1}\dif \theta
\]
is convergent, where $\lambda$ runs through the characters of $T$ appearing as the weights of the linear representation of $H$ into $\Aut(X)$ given by $x=(x_1, \ldots, x_n)\mapsto h x=(h x_1, \ldots, h x_n)$, $m_{\lambda}$ is the multiplicity of $\lambda$, and $\Delta_{P_0}$ is the algebraic module of $P_0$ (note that $|\Delta_{P_0}(\cdot)|_{\BA}^{-1}$ is the modular character of $P_0(\BA)$). The proof of this result depends on Lemma \ref{Lem 2.3} and the choice of the place $v_0$.

In Section \ref{Section 4}, we will study the Siegel Eisenstein series. For $\Phi \in \CS(X(\BA))$ and $s\in \BC$, we define the Siegel Eisenstein series on $G(\BA)$ by
\[
E(g,s,\Phi)=\sum_{\gamma \in P(k)\bs G(k)} f_{\Phi}^{(s)}(\gamma g),
\]
where $f_{\Phi}^{(s)}(g)=|a(g)|^{s-s_0}\omega(g)\Phi(0)$ and $s_0=(m-n)/2$.  We will show in Theorem \ref{Thm 1}, using Godement's convergence criterion in the function field case as established in \cite{Morris 1982}, that $E(g,s, \Phi)$ is absolutely convergent for all $\Phi$ whenever $\Re(s)>n/2$; in particular, it follows  that $E(g,s, \Phi)$ is holomorphic at $s_0$ for all $\Phi$ whenever $m>2n$. Assume $m>2n$ and write $E(\Phi)=E(1, s_0, \Phi)$ for $\Phi \in \CS(X(\BA))$. Using the Bruhat decomposition $G(k)=\cup_{r=0}^n P(k) w_r N(k)$, we can show that $E=\sum_{r=0}^n E_{X_r}$, where $E_{X_0}(\Phi)=\Phi(0)$, and for $r\geq 1$, $E_{X_r}$ is given by
\[
E_{X_r}(\Phi)=\sum_{b\in \Her_n(k)}\omega(w_r n(b))\Phi(0)=\sum_{b\in \Her_n(k)}\int_{X_r(\BA)}\Phi(x)\psi(\tr(x, x)b)\dif x.
\]
In fact, if we identify $V^r$ with $X_r$ via $(x_1, \ldots, x_r)\mapsto (x_1, \ldots, x_r, 0, \ldots, 0)$, we have
\[
E_{X_r}(\Phi)=\sum_{b\in \Her_r(k)}\int_{V^r(\BA)}\Phi(x)\psi(\tr(x, x)b)\dif x.
\]
We will regard $E_{X_r}$ as coming from the pair $(G_r, H)$, where $G_r=U(r, r)$.
Integrals of the form $F_{\Phi}^{*}(b):=\int_{X(\BA)}\Phi(x)\psi(\tr(x, x)b)\dif x$ are studied in \cite[n. 2]{Weil 1965}; and by using the results proved there, we can show that $E_{X}(\Phi)=\sum_{b\in \Her_n(k)}F_{\Phi}^{*}(b)=\sum_{b\in \Her_n(k)}F_{\Phi}(b)$, where $F_{\Phi}$ is the Fourier transform of $F_{\Phi}^{*}$. So we have $E_X=\sum_{b\in \Her_n(k)}\mu_b$, where $\mu_b$ is given by $\mu_b(\Phi)=F_{\Phi}(b)$ for $\Phi \in \CS(X(\BA))$.
Moreover, we can show as in \cite{Weil 1965} that each $\mu_b$ is given by the positive measure $|\theta_b|_{\BA}$ determined by a gauge form $\theta_b$ on the variety $U(b)$, where $U(b)$ consists of points $x=(x_1, \ldots, x_n)$ in $X$ of maximal rank and satisfying $i_X(x):=((x_i, x_j))=b$.
These results will be summarized as two theorems: Theorem \ref{Thm 2} says that $E_X=\sum_{b\in \Her_n(k)}\mu_b$, where each $\mu_b$ is given by the positive measure $|\theta_b|_{\BA}$ determined by a gauge form $\theta_b$ on the variety $U(b)$; and Theorem \ref{Thm 3} says that $E=E_X+\sum_{0\leq r \leq n-1}E_{X_r}$, where each $E_{X_r}$ has the same form as $E_X$.

In Section \ref{Section 5}, we will prove some uniqueness theorems analogous to those in \cite{Weil 1965}. We say a positive measure is tempered if it is defined by a positive tempered distribution, and we identify a positive tempered measure with the positive tempered distribution used to define it. Let $\hat{E}$ be a tempered measure on $X(\BA)$, invariant under $G(k)$, and let $\Phi \in \CS(X(\BA))$; then
$g\mapsto \hat{E}(\omega(g)\Phi)$ is a continuous function on $G(\BA)$, left invariant under $G(k)$. We will give conditions for this function to be bounded on $G(\BA)$, uniformly in $\Phi$ on every compact subset of $\CS(X(\BA))$; for this,
we will use the results of the reduction theory formulated in Section \ref{Section 3} and Lemma \ref{Lem 6}. After giving these conditions, we will prove Theorem \ref{Thm 4}, which says that if $m>2n$, then any positive measure $E'$ on $X(\BA)$, which is invariant under $G(k)$ and the local group $H_v$ for a place $v$ of $k$ such that $U(0)_v$ is non-empty, and such that $E'-E$ is a sum of measures supported by $U(b)_{\BA}$ for any $b\in \Her_n(k)$, must be equal to $E$ itself, where $E$ is the positive tempered measure on $X(\BA)$ given by the Siegel Eisenstein series $E(\Phi)$. Here we say, following Weil, that a measure on $X(\BA)$ is supported by $U(b)_{\BA}$ if it is the image of a measure on $U(b)_{\BA}$ under $j_{\BA}$, where $j_{\BA}$ is the canonical injection of $U(b)_{\BA}$ into $X(\BA)$; for example, each $\mu_b$ as above is such a measure. The proof of this theorem is similar to that as in the number field case in \cite{Weil 1965}.

Finally, in Section \ref{Section 6}, we will give a convergence criterion for the theta integral, and prove the Siegel-Weil formula. For $\Phi \in \CS(X(\BA))$, define the theta integral $I(\Phi)$ by
\[
I(\Phi)=\int_{H(\BA)/H(k)} \sum_{x\in X(k)}\Phi(hx)~\dif h,
\]
where $\dif h$ is the Haar measure on $H(\BA)$ such that $\vol(H(\BA)/H(k))=1$.
We will show in Proposition \ref{Prop 8} that the theta integral $I(\Phi)$ is absolutely convergent for any $\Phi \in \CS(X(\BA))$ whenever $r=0$ or $m-r>n$. By Lemma \ref{Lem 5}, it is sufficient to show that the integral
\[
\int_{\Theta^{+}}\prod_{\lambda} \sup(1, |\lambda(\theta)|_{\BA}^{-1})^{m_{\lambda}}\cdot |\Delta_{P_0}(\theta)|_{\BA}^{-1} \dif \theta
\]
is convergent whenever $r=0$ or $m-r>n$; and this can be achieved by direct computations. Note that if $m>2n$, then both the theta integral $I(\Phi)$ and the Siegel Eisenstein series $E(\Phi)$ are absolutely convergent for any $\Phi \in \CS(X(\BA))$, and hence give two positive tempered measures $I$ and $E$ on $X(\BA)$. We will show in Theorem \ref{Thm 5} that if $m>2n$, then $I=E$ and $I_b=\mu_b$ for every $b\in \Her_n(k)$, where $I_b$ is the measure on $X(\BA)$ given by
\[
I_b(\Phi)=\int_{H(\BA)/H(k)}\sum_{\xi \in U(b)_k}\Phi(h \xi)~\dif h.
\]
Here $U(b)_k$ is the set of elements $x$ in $X(k)$ of maximal rank and satisfying $i_X(x)=b$. This theorem will be proved by induction on the restriction of $\Phi$ to $X_r$, following Weil's approach. If $n=0$, then $X=\{0\}$ and $G=\{1\}$, and hence $E(\Phi)=\Phi(0)$ and $I(\Phi)=\vol(H(\BA)/H(k))\cdot \Phi(0)$; the desired results follow from the hypothesis $\vol(H(\BA)/H(k))=1$. Now we assume that $n\geq 1$ and that the results are valid for any $r<n$. We let $I_X=\sum_{b\in \Her_n(k)}I_b$; then
\[
I_X(\Phi)=\int_{H(\BA)/H(k)}\sum_{\xi }\Phi(h \xi)~\dif h,
\]
where $\xi$ runs over the elements in $X(k)$ which are of maximal rank. We define $I_{X_r}$ similarly for $0\leq r \leq n-1$. Then $I=I_X+\sum_{0\leq r \leq n-1}I_{X_r}$. Remember that Theorem \ref{Thm 2} says $E_{X}=\sum_{b\in \Her_n(k)}\mu_b$ and Theorem \ref{Thm 3} says $E=E_X+\sum_{0\leq r \leq n-1}E_{X_r}$. The induction hypothesis implies that $I_{X_r}=E_{X_r}$ for $0\leq r \leq n-1$; thus $I=I_X+E-E_X$, and $I-E=I_X-E_X=\sum_{b\in \Her_n(k)}I_b-\sum_{b\in \Her_n(k)}\mu_b$. Note that the measures $\mu_b$ are supported by $U(b)_{\BA}$,  and it can be shown as in \cite{Weil 1965} that the measures $I_b$ are also so.
Therefore the positive tempered measure $I$ has the property stated in Theorem \ref{Thm 4} and is hence equal to the measure $E$, and thus $I_X=E_X$. Since $I_{b}$ and $\mu_b$ are the restrictions of $I_{X}$ and of $E_X$ to the set $i_X^{-1}(\{b\})$ respectively, it follows that $I_{b}=\mu_b$ for any $b\in \Her_n(k)$. The desired results are thus proved.\\

More generally, we will consider the casea where $D$ is a finite dimensional division algebra over $k$ equipped with an involution $\xi \mapsto \bar{\xi}$, whose center is either $k$ or a quadratic field extension of $k$. We consider similar spaces and groups over $k$ as above.

Let $\eta=\pm 1$.
We consider the space $W=M_{1\times 2n}(D)$ of row vectors of length $2n$ over $D$, equipped with a non-degenerate $(-\eta)$-hermitian form $\pair{~, ~}: W \times W \ra D$ given by
\[
\pair{x,y}=x \begin{pmatrix} 0 & 1_n \\ -\eta \cdot 1_n & 0 \end{pmatrix} {}^t\!\bar{y},
\]
where $1_n$ is the identity matrix of size $n$;
and we consider the space $V=M_{m\times 1}(D)$ of column vectors of length $m$ over $D$, equipped with a non-degenerate $\eta$-hermitian form $(~,~): V\times V \ra D$.

Consider the isometry groups $G$ of $W$ and $H$ of $V$ respectively. For example, if $D=k$ and $\eta=1$, then $W$ is a non-degenerate symplectic space over $k$ and $V$ is a non-degenerate quadratic space over $k$, and $G=Sp(W)$ is the associated symplectic group, $H=O(V)$ is the associated orthogonal group. We will consider the Siegel-Weil formula for the pair $(G, H)$, and we will obtain similar results as above. But we want to give a remark. If $D=k$, $\eta=1$ and $m$ is odd, then the Weil representation applies only to the two-fold metaplectic cover $\wt{Sp(W)(\BA)}$ of $Sp(W)(\BA)$ and not to the group $Sp(W)(\BA)$, and in this case we have to consider the Siegel-Weil formula for $(\wt{Sp(W)}, O(V))$. Note that in any other case the Weil representation applies to the isometry group $G$.\\

{\bf Acknowledgments}.
I thank Prof. Ye Tian and Prof. Song Wang for many helpful discussions on Weil's two Acta Math. papers, and I thank Prof. Wee Teck Gan for some helpful suggestions. I thank a referee of Journal of Number Theory  for many corrections and suggestions which greatly improved the quality of this paper.
I also thank the Morningside Center of Mathematics for its hospitality and support during my visits in Beijing. This work was supported by the National Natural Science Foundation of China (Grant \# 11301164 and \# 11671380).

\section{Notation and preliminaries}\label{Section 2}
In this section, we introduce the notation and preliminaries used in this paper.

Let $q$ be a power of an odd prime number and let $\BF_q$ be the finite field of order $q$. Let $k$ be a function field in one variable over $\BF_q$, such that $\BF_q$ is algebraically closed in $k$. Note that $\mathrm{char}(k)\neq 2$ since $q$ is odd.

Let $\BF_q[x]$ be the polynomial ring in one indeterminate over $\BF_q$, with fraction field $\BF_q(x)$, and we regard $k$ as a finite separable extension of $\BF_q(x)$.

Let $O_k$ be the ring of integers of $k$, which is the integral closure of $\BF_q[x]$ in $k$; it is a Dedekind domain.

A place of $k$ is an equivalence class of nontrivial absolute values on $k$. The ``infinite place" $(x^{-1})$ of $\BF_q(x)$ is defined by the absolute value given by $|f/g|=q^{\deg(f)-\deg(g)}$ for $f, g\in \BF_q[x]$. See $\S 4.4$ of \cite{Ramakrishnan-Valenza 1999} for more details.

For a place $v$ of $k$, let $k_v$ be the completion of $k$ at $v$ (called a local field), $O_v$ be the ring of integers in $k_v$, $\mathfrak{p}_v$ be the maximal ideal of $O_v$, $\mathbb{F}_v$ be the residue field at $v$, and let $q_v$ be the order of the residue field $\mathbb{F}_v$.

Let $\BA$ be the ring of adeles of $k$, and let $\BA^{\times}$ be the group of ideles of $k$. Let $O_{\BA}=\prod_v O_v$, where $v$ runs over the places of $k$.

We fix a place $v_0$ of $k$ which lies above the ``infinite place" $(x^{-1})$ of $\BF_q(x)$ and such that the residue field of $k$ at $v_0$ is $\BF_q$.

We fix an algebraically closed field extension $\Omega$ of $k$, called the universal domain.

Let $D$ be a finite dimensional division algebra over $k$ equipped with an involution $\xi \mapsto \bar{\xi}$. Let $K$ be the center of $D$. We assume that the subfield of $K$ formed of elements invariant under the involution is $k$. Then $K$ is either $k$ or is a quadratic field extension of $k$.

Let $\alpha^2$ be the dimension of $D$ over $K$.

Let $\eta=\pm 1$. Let $\delta$ be the dimension of $D$ over $k$, and let $\delta'$ be the dimension over $k$ of the space of elements $\xi$ of $D$ such that $\bar{\xi}=\eta \xi$, and let $\epsilon=\delta'/\delta$.

Then
\[
\epsilon=
\begin{cases}
0 &\text{if $D=k$ and $\eta=-1$}; \\
\frac{1}{4} &\text{if $D$ is a division quaternion algebra over $k$ and $\eta=1$};\\
\frac{1}{2} &\text{if the center $K$ of $D$ is a quadratic extension of $k$};\\
\frac{3}{4} &\text{if $D$ is a division quaternion algebra over $k$ and $\eta=-1$};\\
1 &\text{if $D=k$ and $\eta=1$}.
\end{cases}
\]
Note that $[K:k]=2$ if and only if $\epsilon=1/2$, and $K=k$ otherwise. See \cite[n. 27]{Weil 1965}.

For a positive integer $d$, let $M_d(D)$ be the additive group of square matrices over $D$ of order $d$. Then $M_d(D)$ is a central simple algebra over $K$.

Let $\nu_K: M_d(D) \ra K$ be the reduced norm, and let $\nu=N_{K/k}\circ \nu_K$. Let $\tau_K: M_d(D) \ra K$ be the reduced trace, and let $\tau=\tr_{K/k}\circ \tau_K$. Note that for $x\in K$, $\nu_K(x)=x^{\alpha}$, where $\alpha^2=[D:K]$. See p. 169 of \cite{Weil 1974}.

For a matrix $x=(x_{ij})$ over $D$, let $\bar{x}=(\bar{x}_{ij})$ be the conjugate of $x$, and let ${}^t\!x=(x_{ji})$ be the transpose of $x$. We often write $x^*$ for the conjugate transpose of $x$, i.e. $x^*={}^t\!\bar{x}$.

Let $n\geq 0$ be an integer. Let $W=M_{1\times 2n}(D)$, and equip it with an $(-\eta)$-hermitian form $\pair{~,~}: W \times W \ra D$ given by
\[
\pair{x,y}=x \begin{pmatrix} 0 & 1_n \\ -\eta \cdot 1_n & 0 \end{pmatrix} y^*
\]
for $x,y\in W$. In this case, we say $W$ is a split $(-\eta)$-hermitian space.

Let
\[
G=\{g\in GL(W): \pair{xg, yg}=\pair{x,y}, \forall x,y\in W\}
\]
be the isometry group of $W$.

Then
\[
G(k)=\{g\in GL_{2n}(D): g \begin{pmatrix} 0 & 1_n \\ -\eta \cdot 1_n & 0 \end{pmatrix} g^* =\begin{pmatrix} 0 & 1_n \\ -\eta \cdot 1_n & 0 \end{pmatrix}\}.
\]

We have the Siegel parabolic subgroup $P=N M$ of $G=U(W)$, where the Levi subgroup $M$ is given by
\[
M(k)=\{m(a)=\begin{pmatrix} a& 0\\ 0 & (a^*)^{-1}\end{pmatrix}: a\in GL_n(D) \},
\]
and the unipotent radical $N$ is given by
\[
N(k)=\{n(b)=\begin{pmatrix} 1_n & b\\ 0 & 1_n \end{pmatrix}: b\in \Her_n(k)\},
\]
where
\[
\Her_n(k)=\{b\in M_{n}(D): b^*=\eta \cdot b\}
\]
 is the space of $\eta$-hermitian matrices of order $n$ over $D$.

We have the Bruhat decomposition:
\[
G(k)=\bigcup_{r=0}^n P(k) w_r P(k)=\bigcup_{r=0}^n P(k) w_r N(k),
\]
where
\[
w_{r}=\begin{pmatrix} 1_{n-r}& &  0 &\\ & 0 & & 1_r\\ 0 & & 1_{n-r}& \\ & -\eta \cdot 1_r& & 0 \end{pmatrix}.
\]

Let $m$ be a positive integer. Let $V=M_{m\times 1}(D)$, and equip it with an $\eta$-hermitian form $(~,~): V\times V \ra D$ given by
\[
(x,y)=x^* \cdot Q\cdot y,
\]
where $Q$ is an invertible element of $M_m(D)$ such that $Q^*=\eta \cdot Q$ (i.e. $\eta$-hermitian).

Denote the Witt index of $V$ by $r$, i.e. $r$ is the dimension over $D$ of a maximal totally isotropic subspace of the $\eta$-hermitian space $V$. We also say the $\eta$-hermitian matrix $Q$ is of Witt index $r$.

Let
\[
H=\{h\in GL(V): (hx, hy)=(x,y), \forall x,y\in V\}
\]
be the isometry group of $V$.

Then
\[
H(k)=\{h\in GL_m(D): h^* \cdot Q\cdot h=Q\}.
\]

If $\epsilon=0$, then $G=O(W)=O(n, n)$ is the orthogonal group associated to the split quadratic space $W$, and $H=Sp(V)=Sp(m)$ is the symplectic group associated to the symplectic space $V$.
If $\epsilon=1/4$, then $G$ is the quaternionic unitary group associated to the split quaternionic skew-hermitian space $W$, and $H$ is the quaternionic unitary group associated to the quaternionic hermitian space $V$. If $\epsilon=1/2$, then we may assume $\eta=1$ by replacing an $\eta$-hermitian matrix $Q$ by $\xi Q$, where $\xi$ is some element in $K$ such that $\xi \bar{\xi}=\eta$ which exists by Hilbert's Theorem 90, and then $G$ is the unitary group associated to the split skew-hermitian space $W$, and $H$ is the unitary group associated to the hermitian space $V$.
If $\epsilon=3/4$, then $G$ is the quaternionic unitary group associated to the split quaternionic hermitian space $W$, and $H$ is the quaternionic unitary group associated to the quaternionic skew-hermitian space $V$. If $\epsilon=1$, then $G=Sp(2n)$ is the symplectic group associated to the symplectic spacc $W$, and $H=O(V)$ is the orthogonal group associated to the quadratic space $V$.

Let $X=M_{m\times n}(D)$ be the additive group of $m\times n$ matrices over $D$ (except in Section 2, where we allow $X$ to be an arbitrary vector space over $k$). We write an element $x$ of $X=M_{m\times n}(D)$ in the form $x=(x_1, \ldots, x_n)$, where each $x_i\in M_{m\times 1}(D)$. Then $X=V^n$ with $V=M_{m\times 1}(D)$. We can regard $X$ as a left module over $M_{m}(D)$ of rank $n$.

For $0\leq r\leq n$, let $X_r$ be the subspace of $X$ consisting of elements of the form $(x_1, \ldots, x_r, 0, \ldots, 0)$. In particular, $X_0=0$ and $X_n=X$.

The group $H$ acts on $X$ via $h\cdot (x_1, \ldots, x_n)=(h\cdot x_1, \ldots, h\cdot x_n)$, where $x_i\in V=M_{m\times 1}(D)$.

Let $\BW=V\otimes_D W$, and equip it with a symplectic form over $k$ given by
\[
\pair{\pair{~,~}}=\kappa \cdot \tr_{D/k}((~,~)\otimes \overline{\pair{~,~}}),
\]
where
\[
\kappa=
\begin{cases}
2 &\text{if $D=k$},\\
1 & \text{otherwise}.
\end{cases}
\]
Note that we have followed the notation on p. 279 of \cite{Kudla-Sweet 1997}, so the $\kappa$ here is twice of that on p. 364 of \cite{Kudla 1994}.

Let
\[
Sp(\BW)=\{u\in GL(\BW):\pair{\pair{xu,yu}}=\pair{\pair{x,y}}, \forall x,y\in \BW \}
\]
be the symplectic group associated to the symplectic space $\BW$. Note that $Sp(\BW)$ acts on $\BW$ on the right.

Then $(G, H)$ is a {\sl reductive dual pair of type I} in $Sp(\BW)$ in Howe's sense (\cite[n. 5]{Howe 1979}).

There is a natural homomorphism
\[
\iota: G \ra Sp(\BW)
\]
given by $(v\otimes w)\iota(g)=v\otimes wg$ for $v\in V$, $w\in W$ and $g\in G$.

Let $i_X: X \ra \Her_n$ be the mapping given by $i_X(x)=(x,x):=((x_i, x_j))$ for $x=(x_1, \ldots, x_n)\in X$ with each $x_i\in V$.

When $b$ belongs to $\Her_n(k)$, the set $i_X^{-1}(\{b\})$, on the universal domain $\Omega$, is a $k$-closed
subset of $X(\Omega)$. We denote by $U(b)$ the set of points of maximal rank of this set; it is a $k$-open
subset of $i_X^{-1}(\{b\})$. Thus $U(b)_k$ is the set of points $x$ in $X(k)$ of maximal rank and satisfying $i_X(x)=b$.

For an algebraic group $\mathfrak{G}$ over $k$ and for a $k$-algebra $A$, we write $\mathfrak{G}(A)$ (or $\mathfrak{G}_A$) for the group of $A$-points. In particular, for a place $v$ of $k$, we often write $\mathfrak{G}_v=\mathfrak{G}(k_v)$, and write $\mathfrak{G}_v^o=\mathfrak{G}(O_v)$ whenever $\mathfrak{G}$ is defined over $O_v$. For example, we write $D_{\BA}=D\otimes_k \BA$, $K_{\BA}=K\otimes_k \BA$, and $X(\BA)=X\otimes_k \BA$.

For a place $v$ of $k$, let $|\cdot |_v$ be the absolute value on $k_v$, and let $|\cdot |_{\BA}=\prod_v |\cdot|_v$ be the adelic absolute value on $\BA$.

Define the adelic absolute value on $K_{\BA}$ by $|x|_{K_{\BA}}=|N_{K/k}(x)|_{\BA}$ for $x\in K_{\BA}$.

Let $\wt{G(\BA)}$ be the two-fold metaplectic cover of $G(\BA)$ if $\epsilon=1$ and $m$ is odd, and let $\wt{G(\BA)}=G(\BA)$ otherwise. Note that if $\epsilon=1$ and $m$ is odd, then $W$ is a symplectic space and $G=Sp(W)$ is the associated symplectic group, and $\wt{G(\BA)}=G(\BA)\times \{\pm 1\}$ with group multiplication given by
\[
(g_1, z_1)\cdot (g_2, z_2)=(g_1 g_2, z_1 z_2 \tilde{c}(g_1, g_2))
\]
where $\tilde{c}(g_1, g_2)$ is Rao's normalized $\{\pm 1\}$-valued cocycle defined locally as in Thm 5.3 on p. 361 of \cite{Rao 1993}; and we can embed $G(k)$ and $N(\BA)$ into $\wt{G(\BA)}$ via $g \mapsto (g, 1)$.

Let $\pi: \wt{G(\BA)} \ra G(\BA)$ be the canonical projection. Let $\wt{P(\BA)}=\pi^{-1}(P(\BA))$, and let $\wt{M(\BA)}=\pi^{-1}(M(\BA))$. Note that $\wt{P(\BA)}=N(\BA) \wt{M(\BA)}$, where $N(\BA) $ can be embedded  into $\wt{P(\BA)}$ via $n(b)\mapsto (n(b), 1)$.

For the group $G$, there is an Iwasawa decomposition $G(\BA)=N(\BA) M(\BA)G(O_{\BA})$. Let $\wt{G(O_{\BA})}=\pi^{-1}(G(O_{\BA}))$; then $\wt{G(\BA)}=N(\BA) \wt{M(\BA)} \wt{G(O_{\BA})}$.

For $g\in G(\BA)$, write $g=n(b)m(a)g_1$ with $n(b)\in N(\BA)$, $m(a)\in M(\BA)$ and $g_1\in G(O_{\BA})$, and let $|a(g)|=|\nu_K(a)|_{K_{\BA}}=|\nu(a)|_{\BA}$. For $\tilde{g}=(g, \zeta)\in \wt{G(\BA)}$, let $|a(\tilde{g})|=|a(g)|$.

Fix a non-trivial additive character $\psi: \BA/k \ra \BC^{\times}$.

We identify $\Her_n(\BA)$ with its Pontryagin dual via $\psi$:
\[
[b_1, b_2]\mapsto \psi(\frac{\kappa}{2}\cdot \tau(b_1 b_2)), \qquad \text{for $b_1, b_2\in \Her_n(\BA)$}.
\]

For the local groups $Sp(\BW)_v$ and the adelic group $Sp(\BW)_{\BA}$, there are associated metaplectic groups $Mp(\BW)_v$ and $Mp(\BW)_{\BA}$:
\[
\begin{aligned}
& 1\ra \BC^1 \ra Mp(\BW)_{v} \ra Sp(\BW)_v \ra 1,\\
& 1\ra \BC^1 \ra Mp(\BW)_{\BA} \ra Sp(\BW)_{\BA} \ra 1,
\end{aligned}
\]
here $\BC^1=\{z\in \BC: z \bar{z}=1\}$. See \cite[n. 34, 37]{Weil 1964} and \cite{Rao 1993}.

The local metaplectic group $Mp(\BW)_{v}$ can be identified with the set $Sp(\BW)_{v}\times \BC^1$ equipped with a group multiplication
\[
(g_1, z_1)\cdot (g_2, z_2)=(g_1 g_2, z_1 z_2 \cdot c_v(g_1, g_2)),
\]
where $c_v(g_{1}, g_{2})$ is Rao's cocycle as in Thm. 4.1 on p. 358 of \cite{Rao 1993}.

Similarly, the adelic metaplectic group $Mp(\BW)_{\BA}$ can be identified with the set $Sp(\BW)_{\BA}\times \BC^1$ equipped with a group multiplication
\[
(g_1, z_1)\cdot (g_2, z_2)=(g_1 g_2, z_1 z_2 \cdot c(g_1, g_2)),
\]
where $c(g_1, g_2)=\prod_v c_v(g_{1v}, g_{2v})$.

For the non-trivial additive character $\psi: \BA/k \ra \BC^{\times}$, there is an associated Weil representation $\omega_{\psi}$ of $Mp(\BW)_{\BA}$ on $\CS(X(\BA))$ given by
\[
\omega_{\psi}(g,z)=z\cdot r_{\psi}(g),
\]
where
\[
r_{\psi}(g)=\otimes r_{\psi_v}(g_v)
\]
and $r_{\psi_v}(g_v)$ is defined as in Thm. 3.5 on p. 355 of \cite{Rao 1993}.\\

Now we fix a Hecke character $\chi$ of $K$ as follows. Recall $K$ is the center of $D$, and $K$ is a quadratic field extension of $k$ if $\epsilon=1/2$, and $K=k$ otherwise.

If $\epsilon=0$ or $\frac{1}{4}$, let $\chi$ be the trivial character of $\BA^{\times}$.

If $\epsilon=\frac{1}{2}$, let $\chi$ be a unitary character of $K_{\BA}^{\times}/K^{\times}$ with $\chi|_{\BA^{\times}}=\epsilon_{K/k}^{\alpha m}$, where $\epsilon_{K/k}$ is the quadratic Hecke character of $k$ associated to the quadratic field extension $K/k$ by class field theory, and $\alpha^2$ is the dimension of $D$ over its center $K$.

If $\epsilon=\frac{3}{4}$, let $\chi$ be the quadratic character of $\BA^{\times}/k^{\times}$ given by $\chi(x)=\prod_v (x_v, (-1)^m \det V)_v$, where $\det V\in k^{\times}/k^{\times 2}$ is the reduced norm of the matrix $((e_i, e_j))_{1\leq i, j\leq m}$ for any basis $\{e_1, \ldots, e_m\}$ of $V$ over $D$.

If $\epsilon=1$, let $\chi$ be the quadratic character of $\BA^{\times}/k^{\times}$ given by $\chi(x)=\prod_v (x_v, (-1)^{m(m-1)/2}\det V)_v$, where $(,)_v$ is the Hilbert symbol for $k_v$, and $\det V\in k^{\times}/k^{\times 2}$ is the determinant of the matrix $((e_i, e_j))_{1\leq i, j\leq m}$ for any basis $\{e_1, \ldots, e_m\}$ of $V$ over $k$.\\

We regard $\chi$ as a character of $\wt{P(\BA)}$ as follows. If $\epsilon=1$ and $m$ is odd, then $\wt{P(\BA)}$ is a double cover of $P(\BA)$, and we let $\chi((n(b)m(a), z))=z \cdot \chi(\det a) \cdot \prod_v \gamma_v(\det a_v, \psi_v)^{-1}$, where for a character $\eta$ on the local field $k_v$ and for $\alpha \in k_v^{\times}$,  $\gamma_v(\alpha, \eta)=\frac{\gamma_v(\alpha \eta)}{\gamma_v(\eta)}$, where $\alpha \eta$ is the character $x\mapsto \eta(\alpha x)$, and $\gamma_v(\eta)$ is the Weil index of the character of second degree $x\mapsto \eta(x^2)$ on $k_v$ (\cite[Appendix]{Rao 1993}). If $\epsilon\neq 1$ or $m$ is even, then $\wt{P(\BA)}=P(\BA)$, and we let $\chi(n(b)m(a))=\chi(\nu_K(a))$.

With the help of $\chi$, there is a splitting $\tilde{\iota}_{\chi}: \wt{G(\BA)}  \ra Mp(\BW)_{\BA}$ given as follows.

If $\epsilon=1$ and $m$ is odd, then $\wt{G(\BA)}$ is a double cover of $G(\BA)$, and the splitting is given by
\[\begin{aligned}
\tilde{\iota}_{\chi}: \wt{G(\BA)} & \ra Mp(\BW)_{\BA},\\
(g, z) & \mapsto (\iota(g), z \beta_{\chi}(g)),
\end{aligned}\]
where $\beta_{\chi}(g)=\prod_v \beta_{\chi_v}(g_v)\in \BC^1$ with $\beta_{\chi_v}(g_v)$ defined as in Thm. 3.1 on p. 378 of \cite{Kudla 1994}. Note that for $g_1, g_2 \in G(\BA)$, $c(\iota(g_1), \iota(g_2))\beta_{\chi}(g_1 g_2)^{-1}\beta_{\chi}(g_1) \beta_{\chi}(g_2)=\tilde{c}(g_1, g_2)^m=\tilde{c}(g_1, g_2)$, since $m$ is odd and $\tilde{c}(g_1, g_2)\in \{\pm 1\}$. See \cite[p. 379]{Kudla 1994}.

If $\epsilon \neq 1$ or $m$ is even, then $\wt{G(\BA)}=G(\BA)$, and the splitting is given by
\[\begin{aligned}
\tilde{\iota}_{\chi}: G(\BA) & \ra Mp(\BW)_{\BA},\\
g & \mapsto (\iota(g), \beta_{\chi}(g)),
\end{aligned}\]
where $\beta_{\chi}(g)=\prod_v \beta_{\chi_v}(g_v)\in \BC^1$ with $\beta_{\chi_v}(g_v)$ defined as in Thm. 3.1 on p. 378 of \cite{Kudla 1994}.

Then the Weil representation $\omega$ of $\wt{G}(\BA)$, associated to the data $(\psi, \chi)$, is defined to be
\[
\omega=\omega_{\psi, \chi}:=\tilde{\iota}_{\chi} \circ \omega_{\psi}.
\]
If $\epsilon\neq 1$ or $m$ is even, then $\wt{G(\BA)}=G(\BA)$, and the Weil representation is given explicitly as follows:
\[\begin{aligned}
 \omega(m(a))\Phi(x) &=\chi(m(a))|\nu(a)|_{\BA}^{\frac{\alpha m}{2}}\Phi(xa), \\
\omega(n(b))\Phi(x) &=\psi(q_b(x))\Phi(x), \\
\omega(w_n)\Phi(x) &=\int_{X(\BA)}\Phi(y)\psi(\kappa \cdot \tau((x, y)))\, \dif y,\\
 \omega(w_r)\Phi(x) &=\int_{X''(\BA)}\Phi(x'+z)\psi(\kappa \cdot\tau((x'',z)))\,\dif z,
\end{aligned}\]
for $\Phi \in \CS(X(\BA))$, $m(a)\in M(\BA)$, $n(b)\in N(\BA)$, $x=x'+x''\in X(\BA)$ with $x'\in X'(\BA)$ and $x''\in X''(\BA)$, where  $X'=\{(x_1, \ldots, x_{n-r}, 0, \ldots, 0)\}$ and $X''=\{(0,\ldots, 0, x_{n-r+1}, \ldots, x_n)\}$. Here
\begin{itemize}
\item \text{$\chi(m(a))=\chi(\nu_K(a))$};\\
\item \text{$\alpha^2$ is the dimension of $D$ over its center $K$};\\
\item \text{$\nu_K: M_n(D) \ra K$ is the reduced norm};\\
\item \text{$\tau=N_{K/k} \circ \tau_K$, where $\tau_K: M_n(D) \ra K$ is the reduced trace};\\
\item \text{$q_b(x)=\frac{\kappa}{2}\cdot \tau((x,x)b)$, where $\kappa=2$ if $D=k$ and $\kappa=1$ otherwise};\\
\item \text{$(y, x)=((y_i, x_j))\in \Her_n(\BA)$ for $y=(y_1, \ldots, y_n)$ and $x=(x_1, \ldots, x_n)$ in $X(\BA)$};\\
\item \text{$\dif y$ is the self-dual Haar measure on $X(\BA)$}.
\end{itemize}
If $\epsilon=1$ and $m$ is odd, then $\wt{G(\BA)}$ is a double cover of $G(\BA)$, and the action of $(m(a), \zeta)\in \wt{M(\BA)}$  is given by
\[
\omega((m(a), z))\Phi(x)=\chi((m(a), z))|\det a|_{\BA}^{\frac{m}{2}}\Phi(xa),
\]
where $\chi((m(a), z))=z \cdot \chi(\det a) \cdot \prod_v \gamma_v(\det a_v, \psi_v)^{-1}$. The actions of $u(b)\in N(\BA)$ and $w_r$ are the same as above, once we embed $N(\BA)$ and $G(k)$ into $\wt{G(\BA)}$ via $g\mapsto (g, 1)$.

The Weil representation of the group $H(\BA)$ on $\CS(X(\BA))$ is given linearly: $\omega(h)\Phi(x)=\Phi(h^{-1}x)$, where for $x=(x_1, \ldots, x_n)$ with each $x_i\in V(\BA)$, $h^{-1}x=(h^{-1}x_1, \ldots, h^{-1}x_n)$.

For the above formulas, see for example \cite[p. 400]{Kudla 1994}, \cite[p. 38]{Kudla 1996} and \cite[p. 280]{Kudla-Sweet 1997} in the local case. Note that in the references \cite{Rao 1993}, \cite{Kudla 1996} and \cite{Kudla-Sweet 1997} the base field is assumed to be of odd characteristic, so the usual actions of the Weil representation in the characteristic zero case are also valid in the odd characteristic case.

Let $\BG_m=GL_1$ be the multiplicative group in one variable over $k$.

For a locally compact group $\mathbf{G}$, its {\sl modular character} $\delta_{\mathbf{G}}: \mathbf{G} \ra \BR^{+}$ is defined by
\[
\dif \mu(gx)=\delta_{\mathbf{G}}(x)^{-1} \dif \mu(g)
\]
for a left Haar measure $\mu$ on $\mathbf{G}$.

For a connected algebraic group $U$ over $k$, its {\sl algebraic module} $\Delta_{U}$ is a $k$-rational character such that
\[
\omega(a^{-1}xa)=\Delta_{U}(a)\omega(x),
\]
where $\omega$ is any gauge form on $U$ (see \cite[p. 11]{Weil 1965}).
It is easy to check, using the measures determined by the gauge forms, that the modular character of the adelic group $U(\BA)$ can be expressed by the algebraic module $\Delta_U$ of $U$ as follows:
\[
\delta_{U(\BA)}(g)=|\Delta_{U}(g)|_{\BA}^{-1}.
\]

Finally, we say a few words about measures and distributions. We follow the convention on p. 3 of \cite{Weil 1965} and say a positive measure is tempered if it is defined by a positive tempered distribution via the Riesz representation theorem. We identity a positive tempered measure with the positive tempered distribution which is used to define it, and vice versa.

\section{Analytic preliminaries}\label{Section 3}
In this section, we will formulate some results in Harder's reduction theory over function fields (\cite{Harder 1969, Harder 1974}), and prove some  auxiliary lemmas analogous to those in \cite{Weil 1965}.

Throughout this section, we let $\CG$ be either $G$ or $H$, and let $\CG^0$ be the identity component of $\CG$.
By changing the sign of $\eta$ if necessary, we can always assume that $\CG=H$, where $H$ is the isometry group of a non-degenerate $\eta$-hermitian space $V$ over $D$, and we assume that $V$ is of dimension $m$ and of Witt index $r$.

Let $T\cong (\BG_m)^r $ be a maximal split torus in $\CG^0$ given as follows. We may assume by choosing a suitable basis of $V$ that the $\eta$-hermitian form on $V$ is given by the matrix
\[
Q=\begin{pmatrix} 0 & 0 & 1_r \\ 0 & Q_0 & 0 \\ \eta \cdot 1_r & 0 & 0  \end{pmatrix},
\]
where $Q_0$ is the matrix (of order $m-2r$) of an anisotropic $\eta$-hermitian form. For $t=(t_1, \ldots, t_r)\in (\BG_m)^r$, denote by $d(t)$ the diagonal matrix of order $m$ whose diagonal elements are
\[
(t_1, \ldots, t_r, 1, \ldots, 1, t_1^{-1}, \ldots, t_r^{-1}).
\]
Then $d$ is an isomorphism of $(\BG_m)^r$ onto a maximal split torus $T$ of $\CG^0$.

Recall we have fixed a place $v_0$ of $k$ such that the residue field of $k_{v_0}$ is $\BF_q$, and we let $\varpi$ be a uniformizer at $v_0$.

Motivated by the constructions on p. 19 of \cite[$\S$ I.2.1.]{Moeglin-Waldspurger 1995}, we proceed as follows. For $\tau \in \BZ$, define an element $a_{\tau}\in \BA^{\times}$ whose component at $v_0$ is $\varpi^{\tau}$ and the component at any other place is 1. In particular, $|a_{\tau}|_{\BA}=q^{-\tau}$.

Let
\[
\Theta(T)=\{(a_{\tau_1}, \ldots, a_{\tau_r}): \tau_i \in \BZ\},
\]
and regard $\Theta(T)$ as a subset of $T(\BA)$ via $(\BG_m)^r\cong T$.

For example, $\Theta(\BG_m)=\{a_{\tau}: \tau \in \BZ\}\cong \BZ$.

Let $P_0$ be a minimal $k$-parabolic subgroup of $\CG^0$ which contains $T$. Then $\CG^0/P_0$ is isomorphic to a projective variety over $k$, and hence $\CG^0(\BA)/P_0(\BA)\cong  (\CG^0/P_0)(\BA)$ is compact. Note that $\CG(\BA)/\CG^0(\BA)$ is compact (see \cite[p. 571, Prop. 3.2.1]{Conrad 2012}), and  it follows that $\CG(\BA)/P_0(\BA)$ is also compact.

Let
\[
P_0(\BA)^1=\{p\in P_0(\BA): |\lambda(p)|_{\BA}=1, \forall \lambda \in X_k(P_0)\},
\]
where $X_k(P_0)$ is the group of $k$-rational characters of $P_0$.

We can define $T(\BA)^1$ and $\CG^0(\BA)^1$ similarly. Note that $X_k(\CG^0)=\{1\}$, so $\CG^0(\BA)^1=\CG^0(\BA)$, whence $\CG^0(\BA)/\CG^0(k)$ is of finite volume for any Haar measure (see \cite{Harder 1969} or \cite[p. 25]{Oesterle 1984}). Since $\CG(\BA)/\CG^0(\BA)$ is compact, it
follows that $\CG(\BA)/\CG(k)$ is also of finite volume for any Haar measure.

By the compactness theorem in reduction theory (see \cite[p. 46, Korollar 2.2.7]{Harder 1969} or \cite[p. 212]{Springer 1994}), we know that $P_0(\BA)^1/P_0(k)$ is compact.

Moreover, we have the following result, which is a generalization of the classical result $\BA^{\times}/\BA^{\times, 1}\cong \Theta(\BG_m)$, and there is an analogue on \cite[p. 17]{Weil 1965} in the number field case.
\begin{lem} \label{Lem 2.1}
$P_0(\BA)/P_0(\BA)^1\cong \Theta(T)$.
\end{lem}

For $c\in \BR$, let
\[
\Theta(c)=\{\theta\in \Theta(T): |\alpha(\theta)|_{\BA}\leq q^c, \forall \alpha \in \Delta\},
\]
where $\Delta$ is the set of simple roots of $\CG^0$ relative to $T$, which is given by
\[
\Delta=\{x_i- x_{i+1}: 1\leq i \leq r-1\}\cup \{2x_r\}.
\]
where $x_i(t)=t_i$ for $t=(t_1, \ldots, t_r)\in T$. See for example \cite[p. 76]{Weil 1965}.

It is easy to verify the following result.
\begin{lem} \label{Lem 2.2}
For $c\in \BR$, we have
\[
\Theta(c)=\{(a_{\tau_1}, a_{\tau_2}, \ldots, a_{\tau_r})\in \Theta(T): -c/2 \leq \tau_r \leq \tau_{r-1}+c \leq \ldots \leq \tau_2+(r-2)c \leq \tau_1+(r-1)c\}.
\]
In particular,
\[
\Theta(0)=\{(a_{\tau_1}, a_{\tau_2}, \ldots, a_{\tau_r})\in \Theta(T): 0 \leq \tau_r \leq \tau_{r-1} \leq \ldots \leq \tau_2 \leq \tau_1\}.
\]
Moreover, for any $c$, we have
\[
\Theta(c)=(a_{-(2r-1)c/2}, a_{-(2r-3)c/2}, \ldots, a_{-3c/2}, a_{-c/2})\cdot \Theta(0).
\]
\end{lem}

For $c\in \BR$, let
\[
T(c)=\{t\in T(\BA): |\alpha(t)|_{\BA}\leq q^c, \forall \alpha\in \Delta\}.
\]
Then it is easy to see that $T(c)=\Theta(c)\cdot T(\BA)^1$.

The fundamental set theorem in reduction theory (see for example \cite[pp. 211--212]{Springer 1994}) claims that there is a compact subset $C_0$ of $\CG^0(\BA)$ and a constant $c\in \BR$ such that
\[
\CG^0(\BA)=C_0\cdot T(c)\cdot P_0(\BA)^1\cdot \CG^0(k).
\]
Now since $\CG(\BA)/\CG^0(\BA)$ is compact (\cite[Prop. 3.2.1]{Conrad 2012}), there exists a compact subset $C_1$ of $\CG(\BA)$ such that $\CG(\BA)=C_1 \cdot \CG^0(\BA)$, whence
\[
\CG(\BA)=C_1 \cdot C_0\cdot T(c)\cdot P_0(\BA)^1 \cdot \CG^0(k).
\]
Taking $\CC=C_1 \cdot C_0$, which is a compact subset of $\CG(\BA)$, we obtain the fundamental set theorem for $\CG$:
\[
\CG(\BA)=\CC\cdot T(c)\cdot P_0(\BA)^1 \cdot \CG(k).
\]

Furthermore, we have the following results analogous to those in \cite[n. 10]{Weil 1965}.
\begin{lem}\label{Lem 2.3}
(i) There exists a compact subset $C$ of $\CG(\BA)$ such that
\[
\CG(\BA)=C\cdot \Theta(0)\cdot P_0(\BA)^1 \cdot \CG(k).
\]
(ii) Suppose $\CG$ is the isometry group of a split $\eta$-hermitian space $V$. Then there exists a compact subset $C_1$ of $\CG(\BA)$ such that
\[
\CG(\BA)=C_1 \cdot T(0)\cdot \CG(k).
\]
\end{lem}
\begin{proof}
(i) It is easy to check that
\[
T(c)\cdot P_0(\BA)^1=\Theta(c) \cdot P_0(\BA)^1.
\]
Thus it follows from the fundamental set theorem for $\CG$ that there is a compact subset $C$ of $\CG(\BA)$ and a constant $c$ such that
\[
\CG(\BA)=C\cdot \Theta(c)\cdot P_0(\BA)^1 \cdot \CG(k).
\]
Replacing $C$ with $C\cdot (a_{-(2r-1)c/2}, a_{-(2r-3)c/2}, \ldots, a_{-3c/2}, a_{-c/2})$, we may always assume that $c=0$.

(ii) In this case, $\CG^0$ is quasi-split, and we can take $P_0$ to be a Borel subgroup such that $T$ is its Levi subgroup. The desired result then follows from (i).
\end{proof}

Now we need the following result, which is an analogue of Lem. 1 on p. 217 of \cite{Godement 1963} in the number field case.
\begin{lem}\label{Lem 2.4}
If $C$ is a compact subset of $P_0(\BA)$, then the union of $\theta C \theta^{-1}$ for $\theta \in \Theta(0)$ is relatively compact in $P_0(\BA)$, i.e. its closure is compact in $P_0(\BA)$.
\end{lem}
\begin{proof}
We follow Godement's method.

Note that $P_0=Z(T)\cdot U$, where $Z(T)$ (the centralizer of $T$ in $\CG^0$) is the Levi subgroup of $P_0$ and $U$ is the unipotent radical of $P_0$. Moreover, the Lie algebra $\text{Lie}(U)$ of $U$ is given by
\[
\text{Lie}(U)=\bigoplus_{\alpha \in \Phi^{+}}\mathfrak{g}_{\alpha},
\]
where $\Phi^{+}$ is the set of positive roots of $\CG^0$ relative to $T$, $\mathfrak{g}$ is the Lie algebra of $\CG^0$, and $\mathfrak{g}_{\alpha}=\{X\in \mathfrak{g}: \text{Ad}(t)X=\alpha(t)X, \forall t\in T\}$ is the root space. See \cite[p. 234]{Borel 1991}.

Take $p\in C$ and $\theta \in \Theta(0)$. Then $p$ and $\theta p \theta^{-1}$ are equal at any place other than $v_0$. Thus it suffices to assume everything is at the place $v_0$. Write $p=zu$, where $z$ belongs to a compact subset of $Z(T)_{v_0}$ and $u$ belongs to a compact subset of $U_{v_0}$. As $\theta p \theta^{-1}=z\cdot \theta u \theta^{-1}$, it suffices to consider $\theta u \theta^{-1}$. Note that similar to the exponential map in the characteristic-zero case,  there is a $T$-equivariant isomorphism $e$ of $\text{Lie}(U)$ onto $U$  (see \cite[p. 184]{Borel 1991}). Write $u=e(X)$, where $X\in \text{Lie}(U)=\oplus_{\alpha \in \Phi^{+}}\mathfrak{g}_{\alpha}$. It suffices to show that if $X$ stays in a fixed compact subset of $\text{Lie}(U)_{v_0}$, then so is $\text{Ad}(\theta)X$. It comes down to assuming $X\in \mathfrak{g}_{\alpha, v_0}$, and then $\text{Ad}(\theta)X=\alpha(\theta)X$. But $\alpha(\theta)$ is a monomial with positive integer exponents in $\alpha_i(\theta)$, where $\{\alpha_i\}\subset \Phi^{+}$ is the set of simple roots,  thus remains bounded on $\Theta(0)\cap T_{v_0}$. The desired result follows.
\end{proof}

Let $\dif g$ be a Haar measure on $\CG(\BA)$. To study the convergence of integrals on $\CG(\BA)/\CG(k)$, we will rely on the following lemma, which is an analogue of Lem. 4 on p. 18 of  \cite[n. 11]{Weil 1965}. Let $\Delta_{P_0}$ be the algebraic module of $P_0$. We write $\Theta^{+}=\Theta(0)$ and $P_0(\BA)^{+}=\Theta^{+}\cdot P_0(\BA)^1$.

\begin{lem}\label{Lem 4}
There exists a compact subset $C_0$ of $\CG(\BA)$ and a constant $\gamma >0$ such that
\begin{equation}\label{12}
\int_{\CG(\BA)/\CG(k)}|F(g)| \,\dif g\leq \gamma \int_{\Theta^{+}}F_0(\theta)\cdot |\Delta_{P_0}(\theta)|_{\BA}^{-1}\,\dif \theta
\end{equation}
whenever $F, F_0$ are locally integrable functions on $\CG(\BA)/\CG(k)$ and on $\Theta^{+}$ respectively, such that $|F(c \theta)|\leq F_0(\theta)$ for all $c \in C_0$ and $\theta \in \Theta^{+}$.
\end{lem}
\begin{proof}
By Lemma \ref{Lem 2.3}, there exists a compact subset $C$ of $\CG(\BA)$ such that $\CG(\BA)=C\cdot P_0(\BA)^{+} \cdot \CG(k)$. Denote by $I$ the first member of (\ref{12}), and by $\varphi_N$ the characteristic function of the set $N=C\cdot P_0(\BA)^{+}$. We have
\begin{equation}\label{13}
I \leq \int_{N/P_0(k)}|F(g)|\,\dif g=\int_{\CG(\BA)/P_0(k)}|F(g)|\varphi_N(g)\,\dif g.
\end{equation}
We will transform the last integral by means of the theory of quasi-invariant measures on homogeneous spaces (see \cite{Bourbaki 1963}, Chap. VII, $\S 2$, n. 5-8). According to this theory, we can construct a continuous function $h$ on $\CG(\BA)$, everywhere $>0$, such that $h(gp)=h(g)|\Delta_{P_0}(p)|_{\BA}$ for all $g\in \CG(\BA)$, $p\in P_0(\BA)$, and then a positive measure $\lambda$ on $\CG(\BA)/P_0(\BA)$ such that for every locally integrable function $f\geq 0$ on $\CG(\BA)/P_0(k)$, we have:
\[
\int_{\CG(\BA)/P_0(k)}f(g)\,\dif g=\int_{\CG(\BA)/P_0(\BA)}\left(h(g)\int_{P_0(\BA)/P_0(k)}f(gp)\,\dif 'p \right)\,\dif \lambda(\dot{g}),
\]
where $\dot{g}$ is the image of $g\in \CG(\BA)$ in $\CG(\BA)/P_0(\BA)$ and  $\dif 'p=|\Delta_{P_0}(\theta)|_{\BA}^{-1}\dif \theta~\dif p_1$ is the right invariant measure on $P_0(\BA)=\Theta(T)\cdot P_0(\BA)^1$, where $\dif \theta$, $\dif p_1$ are Haar measures on $\Theta(T)$ and on $P_0(\BA)^1$ respectively. Applying this formula to the last member of (\ref{13}), we obtain
\[
I \leq \int_{\CG(\BA)/P_0(\BA)}\Psi(\dot{g})\,\dif \lambda(\dot{g}),
\]
where $\Psi$ is the function defined by
\[
\Psi(\dot{g})=h(g)\int_{P_0(\BA)/P_0(k)}|F(gp)|\varphi_N(gp)\,\dif 'p.
\]
Since $\CG(\BA)/P_0(\BA)$ is compact, there is a compact subset $C_1$ of $\CG(\BA)$ such that $\CG(\BA)=C_1 \cdot P_0(\BA)$. We can therefore assume that $g\in C_1$ in the second member of the above formula. But then we have $\varphi_N(gp)=0$ when $p\notin C_1^{-1}N$. Put
\[
Q=C_1^{-1}N\cap P_0(\BA)=(C_1^{-1}C\cap P_0(\BA))\cdot P_0(\BA)^{+}=(C_1^{-1}C\cap P_0(\BA))\cdot \Theta^{+} \cdot P(\BA)^1;
\]
let $\gamma_1$ be the supremum of $h$ on $C_1$, and let $F_1(p)$, for each $p\in P_0(\BA)$, be the supremum of $|F(gp)|$ for $g\in C_1$. Therefore
\[
\Psi(\dot{g})\leq \gamma_1 \int_{Q/P_0(k)}|F(gp)|\,\dif 'p \leq \gamma_1 \int_{Q/P_0(k)} F_1(p)\,\dif 'p,
\]
and consequently, since $\CG(\BA)/P_0(\BA)$ is compact, we have
\[
I \leq \gamma_2 \int_{Q/P_0(k)} F_1(p)\,\dif 'p
\]
provided that the constant $\gamma_2$ is suitably chosen.

By Lemma \ref{Lem 2.1}, we can identify $P_0(\BA)/P_0(\BA)^1$ with $\Theta(T)$. Then it is immediate that every compact subset of $\Theta(T)$ is contained in a set of the form $\theta_0 \Theta^{+}$, with $\theta_0 \in \Theta(T)$. Applying this remark to the image of $C_1^{-1}C\cap P_0(\BA)$ in $P_0(\BA)/P_0(\BA)^1=\Theta(T)$, we conclude that there exists $\theta_0\in \Theta(T)$ such that $Q$ is contained in $\theta_0 \Theta^{+}\cdot P_0(\BA)^1$. On the other hand, since $P_0(\BA)^1/P_0(k)$ is compact by the compactness theorem, there exists a compact subset $C_2$ of $P_0(\BA)^1$ such that  $P_0(\BA)^1=C_2 \cdot P_0(k)$, so we obtain $Q\subset \theta_0 \Theta^{+}\cdot C_2 \cdot P_0(k)$, and consequently
\[
I\leq \gamma_2 \int_{\theta_0 \Theta^{+}\cdot C_2}F_1(p)\,\dif 'p.
\]
Since $\dif 'p=|\Delta_{P_0}(\theta)|_{\BA}^{-1}\dif \theta~\dif p_0$, this can also be written as
\[
I\leq \gamma_2 \int_{C_2} \left(\int_{\Theta^{+}}F_1(\theta_0 \theta p_0)\cdot |\Delta_{P_0}(\theta_0 \theta)|_{\BA}^{-1}\,\dif \theta \right)\,\dif p_0.
\]
Let $C_3$ be the closure of the union of $\theta C_2 \theta^{-1}$ for $\theta \in \Theta^{+}$, which is a compact subset of $P_0(\BA)^1$ by Lemma \ref{Lem 2.4}. Note that $\theta C_2 \theta^{-1}\subset C_3$, and therefore $\theta_0 \theta C_2\subset \theta_0 C_3\theta$ for any $\theta \in \Theta^{+}$. Thus if we denote by $F_2(\theta)$, for any $\theta \in \Theta^{+}$, the supremum of $F_1(p\theta)$ for $p\in \theta_0 C_3$, then we obtain
\[
I\leq \gamma \int_{\Theta^{+}}F_2(\theta)\cdot |\Delta_{P_0}(\theta)|_{\BA}^{-1} \,\dif \theta
\]
provided that the constant $\gamma$ is suitably chosen. It follows that the assertion of the lemma is verified if we take $C_0=C_1 \theta_0 C_3$.
\end{proof}

Now let $X$ be an affine space on which $\CG$ acts via a representation $\rho$ of $\CG$ in $\Aut(X)$. For every  character $\lambda$ of $T$, we denote by $m_{\lambda}$ the dimension over $k$ of the space of vectors $a\in X_k$ such that $\rho(t)a=\lambda(t)a$ for any $t\in T$. The characters $\lambda$ of $T$ for which $m_{\lambda}>0$ are the weights of the representation $\rho$; $m_{\lambda}$ is the multiplicity of the weight $\lambda$.

We have the following analogue of Lem. 5 on p. 20 of \cite[n. 12]{Weil 1965}.
\begin{lem}\label{Lem 5}
Let $\rho$ be a representation of $\CG$ in the group $\Aut(X)$ of automorphisms of an affine space $X$. Then the integral
\begin{equation}
I(\Phi)=\int_{\CG(\BA)/\CG(k)}\sum_{\xi \in X(k)}\Phi(\rho(g)\xi)\cdot \dif g  \label{14}
\end{equation}
is absolutely convergent for any function $\Phi \in \CS(X(\BA))$ whenever the integral
\[
\int_{\Theta^{+}} \prod_{\lambda} \sup(1, |\lambda(\theta)|_{\BA}^{-1})^{m_{\lambda}}\cdot |\Delta_{P_0}(\theta)|_{\BA}^{-1}\, \dif \theta
\]
is convergent, where $\lambda$ runs over the weights of $\rho$; and when this is so, $I(\Phi)$ defines a positive tempered measure $I$. Here we follow the convention on \cite[p. 3]{Weil 1965} and identify a positive tempered distribution on $X(\BA)$ with a positive measure on $X(\BA)$.
\end{lem}
\begin{proof}
If $I(\Phi)$ is absolutely convergent for any function $\Phi \in \CS(X(\BA))$, then Lem. 5 on p. 194 of \cite[n. 41]{Weil 1964} shows that $I(\Phi)$ converges uniformly on every compact subset of $\CS(X(\BA))$, whence it follows, according to Lem. 2 on p. 5 of \cite[n. 2]{Weil 1965}, that $I$ is a tempered distribution, therefore a positive tempered measure. Now let $C_0$ be a compact subset of $\CG(\BA)$ with the property stated in Lemma \ref{Lem 4} above. For $\Phi \in \CS(X(\BA))$,
there exists, according to Lem. 5 of \cite[n. 41]{Weil 1964}, a function $\Phi_1\in \CS(X(\BA))$ such that
\[
|\Phi(\rho(c)x)|\leq \Phi_1(x)
\]
for all $c\in C_0$ and $x\in X(\BA)$. Applying Lemma \ref{Lem 4} to (\ref{14}) then shows that $I(\Phi)$ is absolutely convergent provided that this is so for the integral
\[
I_1=\int_{\Theta^{+}}\sum_{\xi \in X(k)}\Phi_1(\rho(\theta)\xi)\cdot |\Delta_{P_0}(\theta)|_{\BA}^{-1}\,\dif \theta.
\]

We write $X(\BA)=X_{v_0}\times X'$. By the definition of $\CS(X(\BA))$ (see \cite[n. 29]{Weil 1964}), we can assume that $\Phi_1$ is of the form
\[
\Phi_1(x)=\Phi_{v_0}(x_{v_0})\Phi'(x'),
\]
where $x_{v_0}, x'$ are the projections of $x\in X(\BA)$ on $X_{v_0}$ and on $X'$, with $\Phi_{v_0}\in \CS(X_{v_0})$, $\Phi'$ being the characteristic function of a compact open subgroup of $X'$. Let $L$ be the set of $\xi \in X(k)$ whose projection on $X'$ belongs to the support of $\Phi'$. Then $I_1$ can be written as:
\[
I_1=\int_{\Theta^{+}}\sum_{\xi \in L}\Phi_{v_0}(\rho(\theta)\xi)\cdot |\Delta_{P_0}(\theta)|_{\BA}^{-1}\, \dif \theta.
\]

For every weight $\lambda$ of $\rho$, let $X_{\lambda}$ be the subspace of $X_k$, of dimension $m_{\lambda}$ over $k$, formed of eigenvectors of the weight $\lambda$, i.e. vectors $a$ such that $\rho(t)a=\lambda(t)a$ for $t\in T$. Then $X_k$ is the direct sum of $X_{\lambda}$.

Let
\[
(a_{\lambda i})_{1\leq i \leq m_{\lambda}}
\]
be a basis of $X_{\lambda}$ over $k$; replacing $a_{\lambda i}$ by $N^{-1}a_{\lambda i}$ if needed, where $N$ is a suitable element in $O_k$, we may assume that $L$ is contained in the $O_k$-submodule of $X_k$ generated by all the $a_{\lambda i}$. The $a_{\lambda i}$ also form a basis of $X_{v_0}$ over $k_{v_0}$; for $x_{v_0}\in X_{v_0}$, we can thus write
\[
x_{v_0}=\sum_{\lambda, i}x_{\lambda i}a_{\lambda i}
\]
with $x_{\lambda i}\in k_{v_0}$; then, if $\alpha >1$, there is a constant $C$ such that
\[
\Phi_{v_0}(x_{v_0})\leq C \prod_{\lambda, i}(1+|x_{\lambda i}|_{v_0}^{\alpha})^{-1}.
\]

On the other hand, under these conditions, we have
\[
\rho(\theta)x_{v_0}=\sum_{\lambda, i}\lambda(\theta) x_{\lambda i}a_{\lambda i},
\]
where $\lambda(\theta)\in k_{v_0}^{\times}$; and, if $x_{v_0}$ is the projection on $X_{v_0}$ of an element $\xi$ of $L$, then all the $x_{\lambda i}$ are elements in $O_k$ by the choice of the basis $(a_{\lambda i})$. By the choice of $v_0$, we have $|x_{\lambda i}|_{v_0}\geq 1$. Thus we have
\[\begin{aligned}
\sum_{\xi \in L}\Phi_{v_0}(\rho(\theta)\xi)
& \leq C \prod_{\lambda}\left( \sum_{n\geq 0}\frac{1}{1+|\lambda(\theta)|_{v_0}^{\alpha}q^{n\alpha}} \right)^{m_{\lambda}}\\
& \leq C'\prod_{\lambda}\sup(1,|\lambda(\theta)|_{v_0}^{-1})^{m_{\lambda}},
\end{aligned}\]
where $C'$ is a suitable constant.
If we observe that $|\lambda(\theta)|_{\BA}=|\lambda(\theta)|_{v_0}$ for any character $\lambda$ of $T$ and for any $\theta \in \Theta^{+}$, we see that this gives the announced conclusion.
\end{proof}

Finally, we have the following analogue of Lem. 6 on p. 22 of \cite[n. 13]{Weil 1965}. Recall that we have fixed a place $v_0$ of $k$ such that the residue field of $k_{v_0}$ is $\BF_q$, and we denote by $a_{\tau}$, for $\tau \in \BZ$, the idele $(a_v)$ given by $a_v=\varpi^\tau$ for $v=v_0$, and $a_v=1$ for any other place $v$, where $\varpi$ is a uniformizer at $v_0$.

\begin{lem}\label{Lem 6}
 Let $(X^{(\alpha)})_{1\leq \alpha \leq n}$ and $Y$ be vector spaces over $k$; let $X=\prod_{\alpha}X^{(\alpha)}$, and let $p$ be a morphism of $X$ into $Y$, rational over $k$ and such that $p(0,x^{(2)}, \ldots, x^{(n)})=0$ for any $x^{(2)}, \ldots, x^{(n)}$. Let $C_0$ be a compact subset of $\CS(X(\BA))$, and let $N\geq 0$. Then there exists a function $\Phi_0 \in \CS(X(\BA))$ such that
\[
|a_{\tau_1}|_{\BA}^N |\Phi(a_{\tau_1}x^{(1)}, \ldots, a_{\tau_n}x^{(n)})|=q^{-\tau_1 N}|\Phi(a_{\tau_1}x^{(1)}, \ldots, a_{\tau_n}x^{(n)})|\leq \Phi_0(x)
\]
whenever $\Phi \in C_0$, $\tau_1 \leq 0, \ldots, \tau_n\leq 0$, $x=(x^{(1)}, \ldots, x^{(n)})\in X(\BA)$, $p(x)\in Y(k)$ and $p(x)\neq 0$.
\end{lem}
\begin{proof}
Note that a morphism of an affine space into another is just a polynomial mapping. Thus if we choose bases of $X$ and of $Y$ over $k$, then the coordinates of $p(x)$ can be expressed as polynomials with coefficients in $k$ by means of those of $x$. We denote by $d$ the largest degree of these polynomials. On the other hand, write $X(\BA)=X_{v_0}\times X'$, and likewise write $X(\BA)^{(\alpha)}=X_{v_0}^{(\alpha)}\times X'^{(\alpha)}$ and $Y(\BA)=Y_{v_0}\times Y'$; $p$ determines in an obvious way mappings of $X_{v_0}$ into $Y_{v_0}$ and of $X'$ into $Y'$. Choose bases of $X_{v_0}^{(\alpha)}$ and of $Y_{v_0}$ over $k_{v_0}$, and, for $x_{v_0}^{(\alpha)}\in X_{v_0}^{(\alpha)}$ (resp. $y_{v_0} \in Y_{v_0}$), denote by $r_{\alpha}(x_{v_0}^{(\alpha)})$ (resp. $s(y_{v_0})$) the sum of the squares of the absolute value of the coordinates of $x_{v_0}^{(\alpha)}$ (resp. of $y_{v_0}$) with respect to these bases. For $x_{v_0}=(x_{v_0}^{(1)}, \ldots, x_{v_0}^{(n)})\in X_{v_0}$, put
\[
r'(x_{v_0})=\sum_{\alpha \geq 2}r_{\alpha}(x_{v_0}^{(\alpha)}), \quad r(x_{v_0})=r_1(x_{v_0}^{(1)})+r'(x_{v_0}).
\]
Since $p(x)$ vanishes whenever $x^{(1)}=0$, there is a constant $C>0$ such that for any $x_{v_0}\in X_{v_0}$:
\[
s(p(x_{v_0}))\leq C \cdot r_1(x_{v_0}^{(1)}) \cdot r(x_{v_0})^{d-1},
\]
and consequently, for $t_1\geq 1$:
\[
s(p(x_{v_0}))\leq C t_1^{-2}(t_1^2 r_1(x_{v_0}^{(1)})+r'(x_{v_0}))^d.
\]
For $\tau=(\tau_1, \ldots, \tau_n)$, $\tau_i \in \BZ$, and $x_{v_0}\in X_{v_0}$, let
\[
\varpi^{\tau} x_{v_0}=(\varpi^{\tau_1} x_{v_0}^{(1)}, \ldots, \varpi^{\tau_n} x_{v_0}^{(n)});
\]
the inequality which we have obtained shows that, whenever $\tau_1\leq 0$, $\ldots$, $\tau_n\leq 0$:
\[
s(p(x_{v_0}))\leq C q^{-2\tau_1} r(\varpi^{\tau} x_{v_0})^d.
\]
Now, if we apply Lem. 5 on p. 194 of \cite[n. 41]{Weil 1964}, then this shows that we can choose $\Phi_1 \in \CS(X(\BA))$ such that $|\Phi(x)|\leq \Phi_1(x)$ for all $\Phi \in C_0$ and all $x=(x_{v_0}, x')\in X(\BA)$, and likewise we can assume that $\Phi_1$ is of the form
\[
\Phi_1(x)=\Phi_{v_0}(x_{v_0}) \Phi'(x'),
\]
where $\Phi_{v_0}\in \CS(X_{v_0})$ and $\Phi'$ is the characteristic function of a compact open subgroup of $X'$. Let $\mathcal{E}$ be the set of points $x=(x_{v_0}, x')$ of $X(\BA)$ such that $p(x)\in Y(k)$, $p(x)\neq 0$ and $\Phi'(x')\neq 0$; we will show that, on $\mathcal{E}$, $s(p(x_{v_0}))$ has an infimum $\epsilon>0$. In fact, if it is not so, there will be a sequence of points $x_{\nu}=(x_{\nu v_0}, x_{\nu}')$ of $\mathcal{E}$ such that the sequence $p(x_{\nu v_0})$ tends to 0 in $Y_{v_0}$. As the support of $\Phi'$ is compact, we can assume at the same time that the sequence $x_{\nu}'$ tends to a limit $\bar{x}'$, therefore that $p(x_{\nu}')$ tends to $p(\bar{x}')$. But then the sequence of points $y_{\nu}=p(x_{\nu})$ tends to a limit $\bar{y}$ in $Y(\BA)$, for which we have $\bar{y}_{v_0}=0$. As the points $y_{\nu}$ belong to $Y(k)-\{0\}$, which is discrete in $Y(\BA)$, we have $\bar{y}\in Y(k)$, $\bar{y}\neq 0$, therefore $\bar{y}_v\neq 0$ for any $v$, whence the contradiction. Taking into account the inequality proved above, we thus have, for $x\in \mathcal{E}$, $\tau_1\leq 0$, $\ldots$, $\tau_n\leq 0$:
\[
q^{-\tau_1} \leq C' r(\varpi^{\tau} x_{v_0})^{d/2}
\]
with $C'=(C/\epsilon)^{1/2}$.

Now, for any $i\geq 0$, put
\[
a_i=\sup_{x_{v_0}\in X_{v_0}}(r(x_{v_0})^i \Phi_{v_0}(x_{v_0})).
\]
Let $M\geq Nd/2$ be an integer. According to Lem. 4 on p. 193 of \cite[n. 41]{Weil 1964}, there exists $\varphi \in \CS(\BR)$ such that we have, for any $r\in \BR$:
\[
\varphi(x)\geq \inf_{i\geq 0} (a_{M+i}|r|^{-i}).
\]
For $x$ and $\tau$ as above, we thus have, for any $i\geq 0$:
\[\begin{aligned}
q^{-2 \tau_1 M/d} \Phi_{v_0}(\varpi^{\tau} x_{v_0})
&\leq C'^{2M/d} r(\varpi^{\tau} x_{v_0})^M \Phi_{v_0}(\varpi^{\tau} x_{v_0})\\
& \leq C'^{2M/d} a_{M+i} ~r(\varpi^{\tau} x_{v_0})^{-i} \leq C'^{2M/d} a_{M+i} ~r(x_{v_0})^{-i},
\end{aligned}\]
and therefore
\[
q^{-\tau_1 N} \Phi_{v_0}(\varpi^{\tau} x_{v_0})\leq C'^{2M/d} \varphi(r(x_{v_0})).
\]
Thus the conditions of the lemma will be satisfied by setting
\[
\Phi_0(x)=C'^{2M/d} \varphi(r(x_{v_0})) \Phi'(x').
\]
\end{proof}

\section{Siegel Eisenstein series}\label{Section 4}
Recall $G$ is the isometry group of a split space $W$ and $P$ is the Siegel parabolic subgroup of $G$. For $\Phi \in \CS(X(\BA))$ and $s\in \BC$, define the {\sl Siegel Eisenstein series} on $\wt{G(\BA)}$ by
\[
E(g, s,\Phi)=\sum_{\gamma \in P(k)\bs G(k)}f_{\Phi}^{(s)}(\gamma g), \quad \quad \forall g\in \wt{G(\BA)},
\]
where $f_{\Phi}^{(s)}(g)=|a(g)|^{s-s_0}\omega(g)\Phi(0)$ and $s_0=\alpha(m-n+1-2\epsilon)/2$. Here $|a(g)|$ is defined in Section \ref{Section 2}, and $\alpha^2$ is the dimension of the division algebra $D$ over its center $K$.

Similar to the number field case, we are interested in the behavior of $E(g,s,\Phi)$ at $s_0$.

Let $P=NM$ be the Levi decomposition, where $N$ is the unipotent radical and $M\cong \Res_{D/k}GL_n$ is the Levi subgroup.

Let $\BG_m=GL_1$. Let $T\cong (\BG_m)^n$ be the maximal split torus in $G$ given by
\[
T=\{t=(t_1, \ldots, t_n): t_i\in \BG_m\},
\]
where $t=(t_1, \dots, t_n)$ means that $t=\mathrm{diag}(t_1, \ldots, t_n)\in \Res_{D/k}GL_n\cong M$.

Let $Z_M$ be the center of $M$. Then $Z_M\cong \mathrm{Res}_{D/k}\BG_m$. We write an element $z$ of $Z_M$ as the form $z=(z_1, \ldots, z_1)$ with $z_1 \in \mathrm{Res}_{D/k}\BG_m$.

Let $\Delta$ be the set of simple roots of $G$ relative to $T$. Then it is well-known that
\[
\Delta=\{x_i-x_{i+1}: 1\leq i \leq n-1\}\cup \{2x_n\}.
\]
where $x_i(t)=t_i$, $(x_i-x_{i+1})(t)=t_i t_{i+1}^{-1}$ and $(2x_n)(t)=t_n^2$ for $t=(t_1, \ldots, t_n)\in T$. See for example \cite[p. 76]{Weil 1965}.

For a subset $I$ of $\Delta$, there is a parabolic subgroup $P_{I}$ defined by $I$ as follows. Let $\Phi^{+}$ be the set of positive roots of $G$ relative to $T$, let $[I]$ be the set of roots which are linear combinations of elements in $I$ and set $\Psi(I)=\Phi^{+}-[I]$. Let $T_{I}$ be the identity component of $\cap_{\alpha \in I} \ker(\alpha)$, and let $M_{I}=Z_G(T_{I})$ be the centralizer of $T_{I}$ in $G$. Then $[I]=\Phi(T, M_I)$ is the set of roots of $M_I$ relative to $T$. Let $U_{\Psi(I)}$ be the unipotent subgroup defined by Prop. 21.9 on p. 232 of \cite{Borel 1991}, whose Lie algebra is $\sum_{\alpha\in \Psi(I)}\mathfrak{g}_{\alpha}$. Then $P_I=U_{\Psi(I)}M_{I}$. See \cite[p. 97]{Morris 1982} or \cite[p. 234]{Borel 1991}.

Note that $P_{\emptyset}\subset P_I \subset P_{\Delta}=G$, where $P_{\emptyset}=P_0$ is a minimal parabolic, and the proper maximal parabolic subgroups are defined by subsets of the form $\Delta-\{\alpha\}$.

For the Siegel parabolic subgroup $P$, it is easy to check the following.
\begin{lem}
$P$ is defined by $\Delta-\{2x_n\}$.
\end{lem}

Let $X_M(\BR)$ be the group of quasi-characters of $M(\BA)\cong GL_n(D_{\BA})$ into $\BR^{+}$, where $\BR^{+}$ is the set of positive real numbers.
Then $X_M(\BR)\cong \BR$, where we identify $\alpha \in X_M(\BR)$ with $r\in \BR$ if $\alpha(g)=|\nu(g)|_{\BA}^{r}$ for $g\in M(\BA)=GL_n(D_{\BA})$.

Recall that for a character $\alpha$ of $T$ and a cocharacter $\beta^{*}$ of $T$, there is a pairing $(\alpha, \beta^{*})\in \BZ$ defined by
\[
\beta^{*}(\alpha(x))=x^{(\alpha, \beta^{*})}, \quad \forall x\in GL_1.
\]
See \cite[p. 115]{Borel 1991}. This pairing can be extended to a pairing
\[
(\,, \,): X(T)_{\BR} \times X_{*}(T)_{\BR} \ra \BR,
\]
where $X(T)$ is the group of characters of $T$, $X_{*}(T)$ is the group of cocharacters of $T$, $X(T)_{\BR}:=X(T)\otimes_{\BZ}\BR$, etc.

In particular, for a root $\alpha$, the corresponding coroot $\alpha^*$ is defined to be the cocharacter of $T$ such that $(\alpha, \alpha^*)=2$.

Recall the (open) Weyl chamber $C_{P_I}$ associated to a parabolic subgroup $P_I$ defined by a subset $I \subset \Delta$ is given by
\[
C_{P_I}=\{\beta \in X_M(\BR): (\beta, \alpha^*)>0, \forall \alpha \in \Delta-I \},
\]
where $\alpha^*$ is the coroot corresponding to $\alpha$, and we identify an element $\beta=r$ of  $X_M(\BR)=\BR$ with an element of $X(T)_{\BR}$ which sends $t=(t_1, \ldots, t_n)\in T$ to $(t_1 \cdots t_n)^r$ if $r\in \BZ$.  See for example line 10 on p. 118 of \cite{Morris 1982}.

\begin{lem}
Identifying $X_M(\BR)$ with $\BR$. The Weyl chamber $C_P$ associated to the Siegel parabolic subgroup $P$ is given by
\[
C_P=\{r\in \BR: r> 0\}.
\]
\end{lem}
\begin{proof}
Recall $P$ is defined by $\Delta-\{2x_n\}$. For $\alpha=2x_n$, $\alpha^*$ is given by $\alpha^*(x)=(1, \ldots, 1, x)$ for $x\in GL_1$, since $\alpha^*$ satisfies $(\alpha, \alpha^*)=2$, i.e. $\alpha(\alpha^*(x))=x^2$. For $\beta=r\in X_M(\BR)=\BR$, $(\beta, \alpha^*)=r$, since $\beta(\alpha^*(x))=\beta(1,\ldots, 1, x)=x^{r}$.
\end{proof}

Recall that the modular character $\delta_{P(\BA)}$ of $P(\BA)$ can be expressed as
\[
\delta_{P(\BA)}(p)=|\Delta_P(p)|_{\BA}^{-1},
\]
where $\Delta_P$ is the algebraic module of $P$ (see the end of Section \ref{Section 2}).

\begin{lem}
The modular character $\delta_{P(\BA)}$ of $P(\BA)$ is given by
\[
\delta_{P(\BA)}(p)=|a(p)|^{\alpha(n+2\epsilon-1)},
\]
where $\alpha^2$ is the dimension of $D$ over its center.
In particular,
\[
\delta_{P(\BA)}(z)=|\nu(z)|_{\BA}^{\alpha(n+2\epsilon-1)}
\]
for $z\in Z_M(\BA)=GL_n(D_{\BA})$.
\end{lem}
\begin{proof}
For $p=n(b)m(a)\in P$, it follows from Lem. 12 on p. 43 of \cite{Weil 1965} that $\Delta_P(p)=\Delta(a)^{-1}$. But $\Delta(a)=\nu(a)^{\alpha(n+2\epsilon-1)}$ by the formula on p. 48 of \cite{Weil 1965}. The desired result follows.
\end{proof}

\begin{lem}
For $z$ in the center of $\wt{M(\BA)}$ and $g\in \wt{G(\BA)}$, we have
\[
f_{\Phi}^{(s)}(zg)=\lambda(z)f_{\Phi}^{(s)}(g),
\]
where
\[
\lambda(z)=\chi(z)|a(z)|^{s+\frac{\alpha(n+2\epsilon-1)}{2}}.
\]

In particular, the real part $\Re(\lambda)$ of $\lambda$ is given by
\[
\Re(\lambda)(z)=|a(z)|^{\Re(s)+\frac{\alpha(n+2\epsilon-1)}{2}}.
\]
\end{lem}
\begin{proof}
This is just an application of the formulas for the Weil representation in Section \ref{Section 2}. Note that $f_{\Phi}^{(s)}(pg)=\chi(p)|a(p)|^{s+\frac{\alpha(n+2\epsilon-1)}{2}}f_{\Phi}^{(s)}(g)$ for $p\in \wt{P(\BA)}$ and  $g\in \wt{G(\BA)}$, i.e. $f_{\Phi}^{(s)}\in \mathrm{Ind}_{\wt{P(\BA)}}^{\wt{G(\BA)}}(\chi|\cdot|^s)$ (normalized induction).
\end{proof}

Now we can prove the following analogue of Thm. 1 on p. 57 of \cite[n. 40]{Weil 1965}.
\begin{thm}\label{Thm 1}
If $\Re(s)>\alpha (n+2\epsilon-1)/2$, then for any $g\in \wt{G(\BA)}$, the series $E(g,s,\Phi)$ is absolutely convergent for all $\Phi \in \CS(X(\BA))$, and uniformly in $\Phi$ on every compact subset of $\CS(X(\BA))$. In particular, if $m> 2n+4\epsilon-2$, then $E(g,s,\Phi)$ is holomorphic at $s_0=\alpha(m-n+1-2\epsilon)/2$.
\end{thm}
\begin{proof}
This follows from Godement's convergence criterion (when $\wt{G(\BA)}=G(\BA)$, see Lem. 2.2 on p. 118 of \cite{Morris 1982}), which asserts that the series $\sum_{\gamma \in P(F)\bs G(F)}f_{\Phi}^{(s)}(\gamma g)$ converges uniformly for $g$ in a compact set provided $\Re(\lambda)-\delta_{P(\BA)} \in C_P$. See Thm. 3 on p. 125 of \cite{Godement 1967} for the number field case, and see \cite[p. 980]{Morris 1983} or Prop II.1.5 on pp. 85--86 of \cite{Moeglin-Waldspurger 1995} when $\wt{G(\BA)}$ is a double cover of $G(\BA)$.

Note that $\Re(\lambda)-\delta_{P(\BA)} \in C_P$ if and only if $\Re(s)+\alpha(n+2\epsilon-1)/2-\alpha(n+2\epsilon-1)>0$, i.e. $\Re(s)> \alpha(n+2\epsilon-1)/2$.
\end{proof}

From now on, we write
\[
E(\Phi)=E(1,s_0,\Phi)
\]
for $\Phi \in \CS(X(\BA))$. Then
\[
E(\Phi)=\sum_{\gamma \in P(F)\bs G(F)}\omega(\gamma)\Phi(0),
\]
the series on the right side being absolutely convergent whenever $m> 2n+4\epsilon-2$.

By the Bruhat decomposition $G(k)=\cup_{r=0}^n P(k)w_r P(k)=\cup_{r=0}^n P(k)w_r N(k)$, we have
\[
E(\Phi)=\Phi(0)+\sum_{r=1}^n \sum_{b\in \Her_n(k)} \omega(w_r n(b))\Phi(0).
\]
The term $\omega(w_r n(b))\Phi(0)$ is given by
\[
\omega(w_r n(b))\Phi(0)=\int_{X_r(\BA)} \omega(n(b))\Phi(x)\,\dif x=\int_{X_r(\BA)} \Phi(x)\psi(q_b(x))\,\dif x,
\]
where $X_r=\{(x_1, \ldots, x_r, 0, \ldots, 0)\}\subset X$.

In particular, the term $\omega(w_n n(b))\Phi(0)$ is given by
\[
\omega(w_n n(b))\Phi(0)=\int_{X(\BA)} \omega(n(b))\Phi(x)\,\dif x=\int_{X(\BA)} \Phi(x)\psi(q_b(x))\,\dif x.
\]

For $1\leq r \leq n$, let
\begin{equation}
E_{X_r}(\Phi)=\sum_{b\in \Her_n(k)} \int_{X_r(\BA)} \Phi(x)\psi(q_b(x))\,\dif x. \label{32}
\end{equation}
Also let $E_{X_0}(\Phi)=\Phi(0)$.

Note that for any $0\leq r \leq n$, if we embed $\Her_r$ into $\Her_n$ via $b_1 \mapsto \begin{pmatrix} b_1 & 0 \\ 0 & 0 \end{pmatrix}$, then
\[
E_{X_r}(\Phi)=\sum_{b\in \Her_r(k)}\int_{X_r(\BA)} \Phi(x)\psi(q_b(x))\,\dif x.
\]

Then
\begin{equation}
E(\Phi)=\Phi(0)+\sum_{1\leq r \leq n}E_{X_n}(\Phi)=E_{X}(\Phi)+\sum_{0\leq r \leq n-1}E_{X_r}(\Phi). \label{33}
\end{equation}

We assume $m>2n+4\epsilon-2$ in the rest of this section. Note that this is just condition $(\mathbf{B})$ on p. 55 of \cite{Weil 1965}. Then by Theorem \ref{Thm 1} the above series (\ref{32}) and (\ref{33}) are
absolutely convergent, uniformly in $\Phi$ on every compact subset of $\CS(X(\BA))$.

For $b\in \Her_n(k)$, let
\[
F_{\Phi}^*(b)=\int_{X(\BA)}\Phi(x)\psi(q_b(x))\, \dif x.
\]
Then
\[
E_X(\Phi)=\sum_{b\in \Her_n(k)} F_{\Phi}^*(b),
\]
and it follows that this series is absolutely convergent, uniformly in $\Phi$ on every compact subset of
$\CS(X(\BA))$. It follows from the Poisson summation formula that
\begin{equation}
E_X(\Phi)=\sum_{b\in \Her_n(k)} F_{\Phi}(b),  \label{34}
\end{equation}
where $F_{\Phi}$ is the Fourier transform of $F_{\Phi}^*$. Moreover, by Prop. 2 on p. 7 of \cite[n. 2]{Weil 1965}, the Fourier transform is given, for each $b\in \Her_n(\BA)$, by
\[
F_{\Phi}(b)=\int \Phi(x)\, \dif \mu_b(x),
\]
where $\mu_b$ is a positive tempered measure on $X(\BA)$, of support contained in $i_X^{-1}(\{b\})$;
and $F_{\Phi}$ and $F_{\Phi}^*$ are continuous and integrable functions on $\Her_n(\BA)$.
Finally, Prop. 2 of \cite[n. 2]{Weil 1965} shows that the second member of (\ref{34}) is absolutely convergent; as $\mu_b$
are positive measures, we conclude, by Lem. 5 on p. 194 of \cite[n. 41]{Weil 1964}, that the second member
converges uniformly on every compact subset of $\CS(X(\BA))$. According to Lem. 2 on p. 5 of \cite[n. 2]{Weil 1965}, this shows that $E_X$
is a positive tempered measure, given by
\[
E_X=\sum_{b\in \Her_n(k)} \mu_b.
\]
where $\mu_b(\Phi)=\int \Phi \,\dif \mu_b=F_{\Phi}(b)$.
Similarly, for $r<n$, $E_{X_r}$
is a positive tempered measure given by
\[
E_{X_r}(\Phi)=\sum_{b\in \Her_r(k)} \mu_b(\Phi_r),
\]
where $\Phi_r=\Phi|_{X_r(\BA)}$ and $\Her_r$ is embedded into $\Her_n$ via $b_1 \mapsto \begin{pmatrix} b_1 & 0 \\ 0 & 0 \end{pmatrix}$.

Finally, we conclude similarly from formula (\ref{33}) that $E$ is a positive tempered measure, given by the sum of the
measures $E_{X_r}$.

It is easy to see that if $\det(b) \neq 0$, then the $b$-th Fourier coefficient of $E(\Phi)$ can be expressed as
\[
E_{b}(\Phi)=\sum_{r=1}^{n}F_{\Phi_r}(b).
\]

Taking for $\Phi$ a function of the form
\[
\Phi(x)=\prod_v \Phi_v(x_v) \quad \quad (x=(x_v)\in X(\BA)),
\]
where the product is over all the places $v$ of $k$, $\Phi_v$ belongs to $\CS(X_v)$ for any $v$,
and $\Phi_v$ is the characteristic function of $X_v^o:=X(O_v)$ for almost all $v$. And we denote by $F_v$ and $F_v^*$, for each $v$, the functions defined on $\Her_n(k_v)$ by the formulas
\[
F_v(b)=\int_{U_v(b)} \Phi_v(x) \,|\theta_b(x)|_v, \quad
F_v^*(b)=\int_{X_v} \Phi_v(x)\psi_v(q_b(x))\, \dif x;
\]
here we write $U_v(b)$ for the variety formed by points of $i_X^{-1}(\{b\})$ of maximal rank
in $X_v$, and $\theta_b$ for the gauge form defined on this variety by the formula (29) on p. 54 of
\cite[n. 37]{Weil 1965}. According to Prop. 6 on p. 54 of \cite[n. 37]{Weil 1965}, $F_v$ and $F_v^*$ are continuous and integrable, and are Fourier transforms of each other. By
the hypotheses made on $\Phi$, we see immediately that $F_v^*$ takes constant value 1 on $\Her_n(k_v)^o$
 for almost all $v$, here $\Her_n(k_v)^o$
denotes the lattice in $\Her_n(k_v)$ generated by an arbitrarily chosen basis $\Her_n(k)^o$ of
$\Her_n(k)$ over $k$.

It is then immediate that, for any $b=(b_v)\in \Her_n(\BA)$, we have
\[
F_{\Phi}^*(b)=\prod_v F_v^*(b_v),
\]
where almost all the factors of the second member are of value 1. We deduce that
\[
\int |F_{\Phi}^*(b)|~\dif b =\prod_v \int |F_v^*(b_v)|~\dif b_v.
\]
In the above equality, the first member is $<+\infty$; it is $\neq 0$ unless $F_{\Phi}^*=0$;
besides, we can always modify a finite number of the functions $\Phi_v$ so as to have $F_{\Phi}^*\neq 0$,
for example by taking $\Phi_v\geq 0$ and $\Phi_v\neq 0$ for any $v$, which implies that $F_{\Phi}\neq 0$ and
consequently $F_{\Phi}^*\neq 0$. As almost all the factors of the second member of the above equality are $\geq 1$,
it follows that the second member is absolutely convergent (in the sense defined in note $(^1)$ on p. 11 of \cite{Weil 1965}, which means that there exists a finite set $S$ of places of $k$ such that all the factors outside of $S$ are defined and nonzero, and the product of all the factors outside of $S$ is absolutely convergent).
We conclude easily that the Fourier transform $F_{\Phi}$ of $F_{\Phi}^*$ is the product of the Fourier transforms
$F_v$ of $F_v^*$, that is, for any $b=(b_v)\in \Her_n(\BA)$, we have
\[
F_{\Phi}(b)=\prod F_v(b_v),
\]
where the product of the second member is absolutely convergent.

If we denote by $\mu_v$ the tempered measure on $\Her_n(k_v)$ determined by the measure $|\theta_{b_v}|_v$
on $U_v(b_v)$, then $F_v(b_v)$ is just $\mu_v(X_v^o)$ whenever $\Phi_v$ is the characteristic
function of $X_v^o$. The above formula thus shows that the product of $\mu_v(X_v^o)$ is absolutely
convergent, and that the measure $\mu_b$ which appears in the above expression of $F_{\Phi}$
is just $\prod \mu_v$.

When $b$ belongs to $\Her_n(k)$, the set $i_X^{-1}(\{b\})$, on the universal domain $\Omega$, is a $k$-closed
subset of $X(\Omega)$. We denote by $U(b)$ the set of points of maximal rank of this set; it is a $k$-open
subset of $i_X^{-1}(\{b\})$; according to Prop. 3 on p. 34 of \cite[n. 22]{Weil 1965}, when $U(b)$ is not empty, it is
an orbit of the group $U(V)$, taking also on the universal domain. We conclude easily from Lem. 8 on p. 28 of \cite[n. 17]{Weil 1965}
that, if $L\supset k$ is a field containing $k$, then the set $U(b)_L$ of points of $U(b)$ which are rational over $L$ is just the set of points of $i_X^{-1}(\{b\})$ in $X(L)$ which are of maximal rank in $X(L)$. In particular,
for $L=k_v$, we see that $U(b)_v$ is just the set $U_v(b)$.

For $b\in \Her_n(k)$, let $\theta_b$ denote the gauge form on the variety $U(b)$ defined by the formula
\[
\theta_b(x)=\left(\frac{\dif x}{\dif i_X(x)} \right)_b,
\]
in the sense on p. 14 of \cite[n. 6]{Weil 1965}.

Recall that a {\sl system of convergence factors} for an algebraic group $\mathfrak{G}$ over $k$ is a sequence of positive real numbers $\lambda=(\lambda_v)$ indexed by the places of $k$ such that the product $\prod \lambda_v \int_{\mathfrak{G}(O_v)}|\omega|_v$ is absolutely convergent in the sense that there exists a finite set $S$ of places of $k$ such that $\int_{\mathfrak{G}(O_v)}|\omega|_v$ is defined and nonzero for $v\notin S$ and the product $\prod_{v\notin S}\lambda_v \int_{\mathfrak{G}(O_v)}|\omega|_v$ is absolutely convergent in the usual sense, where $\omega$ is a gauge form on $\mathfrak{G}$ and $|\omega|_v$ is the corresponding positive measure on $\mathfrak{G}(k_v)$ for every place $v$ of $k$. See \cite[p. 11]{Weil 1965}.

We have the following analogue of Lem. 19 on p. 61 of \cite[n. 43]{Weil 1965}. The proof is similar and we omit it.

\begin{lem}
For every $b\in \Her_n(k)$, 1 is a system of convergence factors for $U(b)$, and we have $\mu_b=|\theta_b|_{\BA}$.
\end{lem}

In summary, we have shown the following results, which are analogues of Thm. 2 and Thm. 3 on pp. 62--63 of \cite[n. 44]{Weil 1965}.

\begin{thm}\label{Thm 2}
Assume that $m>2n+4\epsilon-2$. For $\Phi \in \CS(X(\BA))$, put
\[
E_{X}(\Phi)=\sum_{b\in \Her_n(k)}\int_{X(\BA)} \Phi(x)\psi(q_b(x))\,\dif x.
\]
Then the series of the second member
is absolutely convergent, and $E_{X}$ is a positive tempered measure. Moreover, for each $b\in \Her_n(k)$,
1 is a system of convergence factors for the variety $U(b)$ of points in $i_X^{-1}(\{b\})$ with
maximal rank; and, if $\theta_b$ denotes the gauge form on this variety defined by the formula
\[
\theta_b(x)=\left(\frac{\dif x}{\dif i_X(x)} \right)_b,
\]
then the measure $|\theta_b|_{\BA}$ on $U(b)_{\BA}$ is equal to the positive tempered measure $\mu_b$ on $X(\BA)$ given by $\mu_b(\Phi)=F_{\Phi}(b)$, and we have
\[
E_{X}=\sum_{b\in \Her_n(k)}\mu_b.
\]
In particular, if $n=0$, then $E_{X}(\Phi)=\Phi(0)$.
\end{thm}

\begin{thm}\label{Thm 3}
Assume that $m>2n+4\epsilon-2$. For $\Phi \in \CS(X(\BA))$, put
\[
E(\Phi)=\sum_{\gamma \in P(F)\bs G(F)} \omega(\gamma)\Phi(0).
\]
Then the series of the second member
is absolutely convergent; $E$ is a positive tempered measure; and we have
\[
E=\sum_{r=0}^n E_{X_r},
\]
where $X_r=\{(x_1, \ldots, x_r, 0, \ldots, 0)\}\subset X$, $E_{X_0}(\Phi)=\Phi(0)$, and for $1\leq r \leq n$,
\[
E_{X_r}(\Phi)=
\sum_{b\in \Her_n(k)}\int_{X_r(\BA)} \Phi(x)\psi(q_b(x))\,\dif x .
\]
Note that if we embed $\Her_r$ into $\Her_n$ via $b_1 \mapsto \begin{pmatrix} b_1 & 0 \\ 0 & 0 \end{pmatrix}$ and let $\Phi_r=\Phi|_{X_r(\BA)}$, then
\[
E_{X_r}(\Phi)=\sum_{b\in \Her_r(k)}\int_{X_r(\BA)} \Phi(x)\psi(q_b(x))\,\dif x=\sum_{b\in \Her_r(k)}\mu_b(\Phi_r).
\]
\end{thm}

\section{Uniqueness theorems}\label{Section 5}
In this section, we will prove some results analogous to those in \cite[Chap. V]{Weil 1965}, which will be used in the proof of the Siegel-Weil formula. Recall $X=M_{m\times n}(D)=V^n$. We assume that $m>2n+4\epsilon-2$ in this section. We say a tempered measure
on $X(\BA)$ is invariant under $G(k)$ when it is invariant under $\omega(g)$ for any $g\in G(k)$; we also say that a measure (tempered or not)
on $X(\BA)$ is invariant under an element $h$ of $H(\BA)$ if it is so under the mapping $x\mapsto hx$ of $X(\BA)$ onto
itself. Recall that, by the corollary to Prop. 9 of \cite[n. 51]{Weil 1964}, the automorphisms $\Phi \mapsto \omega(g)\Phi$ and $\Phi(x) \mapsto \Phi(hx)$ of $\CS(X(\BA))$, for $g\in \wt{G(\BA)}$ and $h\in H(\BA)$, are permutable; it is also the same for
the corresponding automorphisms of the space of tempered distributions on $X(\BA)$.

Let $\hat{E}$ be a tempered measure on $X(\BA)$, invariant under $G(k)$, and let $\Phi \in \CS(X(\BA))$; then
$g\mapsto \hat{E}(\omega(g)\Phi)$ is a continuous function on $\wt{G(\BA)}$, left invariant under $G(k)$. We will give conditions for
this function to be bounded on $\wt{G(\BA)}$, uniformly in $\Phi$ on every compact subset of $\CS(X(\BA))$; for this,
we will apply the results of the reduction theory in Section \ref{Section 3} to the group $G$.

We write $x\in X(k)$ in the form $x=(x_1, \ldots, x_n)$, where each $x_i \in V(k)=M_{m\times 1}(D)$. Let $t=(t_1, \ldots, t_n)$ be an element of $(\mathbb{G}_m)^n$, with each $t_i\in \mathbb{G}_m$; we denote by $\lambda_t$ the automorphism of $X$ defined by
the diagonal matrix whose diagonal elements are $t_1, \ldots, t_n$; it can also be written as
\[
x=(x_1, \ldots, x_n) \mapsto x\lambda_t=(x_1 t_1, \ldots, x_n t_n).
\]

For $t\in (\mathbb{G}_m)^n$,
denote by $\bar{\lambda}_t$ the automorphism of $\Her_n$ determined by the automorphism $\lambda_t$ of $X$, which is given by
\[
b=(b_{\alpha \beta})\mapsto b \bar{\lambda}_t=(b_{\alpha \beta}t_{\alpha} t_{\beta}).
\]
Then the determinants of $\lambda_t$
and of $\bar{\lambda}_t$, with respect to the bases of $X(k)$ and of $\Her_n(k)$ over $k$, are respectively
\[
D(\lambda_t)=(t_1 \ldots t_n)^{m\delta }, \quad
D(\bar{\lambda}_t)=(t_1 \ldots t_n)^{(n+2\epsilon-1)\delta},
\]
where we recall that $\delta$ is the dimension of $D$ over $k$.
We conclude that the gauge form $\theta_b(x)$ on $U(b)$, defined as in Theorem \ref{Thm 2}, is
transformed by $\lambda_t$ to the gauge form
\begin{equation}
\theta_b(x\lambda_t^{-1})=(t_1 \ldots t_n)^{(-m+n+2\epsilon-1)\delta} \theta_{b'}(x) \label{36}
\end{equation}
on $U(b')$, with $b'=b \bar{\lambda}_t$. See \cite[p. 66]{Weil 1965}.

In particular, for $t\in (\BA^{\times})^n$, $\lambda_t$ and $\bar{\lambda}_t$ are automorphisms of $X(\BA)$ and of $\Her_n(\BA)$,
respectively. For $t=(t_1, \ldots, t_n)\in (\BG_m)^n$, if we put $|t|_{\BA}=|t_1 \cdots t_n|_{\BA}$, and regard $\lambda_t$ as the element $\mathrm{diag}(t_1, \ldots, t_n)$ in $GL_n(\BA)$, then we have, whenever $\wt{G(\BA)}=G(\BA)$, for $\Phi \in \CS(X(\BA))$, $x=(x_1, \ldots, x_n)$:
\begin{equation}
\omega(m(\lambda_t))\Phi(x)=\chi(m(\lambda_t))|t|_{\BA}^{m\delta/2}\Phi(x_1 t_1, \ldots, x_n t_n). \label{37}
\end{equation}
When $\wt{G(\BA)}$ is a double cover of $G(\BA)$, we have a similar formula for the action of $\omega((m(\lambda_t), 1))$, which we also write simply as $\omega(m(\lambda_t))$.

We denote by $T$
the image of $(\mathbb{G}_m)^n$ in $G$ under $t\mapsto m(\lambda_t)=\begin{pmatrix} \lambda_t & 0\\ 0 & \lambda_t^{-1}\end{pmatrix}$, where we regard $\lambda_t=\diag(t_1, \ldots, t_n)\in GL_n$; then
$T$ is a maximal split torus of $G$. The strictly positive roots of $G$ relative to $T$ are
$x_{\alpha}-x_{\beta}$ and $x_{\alpha}+x_{\beta}$ for $1\leq \alpha <\beta \leq n$, together with
$2x_{\alpha}$ for $1\leq \alpha \leq n$ in the case where $\epsilon>0$ (see p. 66 of \cite[n. 47]{Weil 1965}). Recall we have defined $\Theta(T)$ and $\Theta^{+}=\Theta(0)$ in Section \ref{Section 3}. Let $T(\BA)^{+}=\Theta^{+}\cdot T(\BA)^1$ . By Lemma \ref{Lem 2.3} there
is a compact subset $C_1$ of $G(\BA)$ such that $G(\BA)=C_1 \cdot T(\BA)^{+} \cdot G(k)$.
Let $T(\BA)'$ be the subset of $T(\BA)$ formed of elements $m(\lambda_t)$ of $T(\BA)$ for which
\[
|t_1|_{\BA}\geq \ldots \geq|t_n|_{\BA}\geq 1.
\]
For $\epsilon>0$, we verify easily that $T(\BA)^{+}=T(\BA)'^{-1}$, so that by putting $C=C_1^{-1}$ we have
$G(\BA)=G(k)\cdot T(\BA)' \cdot C$. If $\epsilon=0$, we verify easily that $T(\BA)^{+}$ is the union of $T(\BA)'^{-1}$ and $s_1^{-1} T(\BA)'^{-1}s_1$, where
\[
s_1=\begin{pmatrix}
1_{n-1} & 0 & 0 & 0\\
0 & 0 & 0 & 1\\
0 & 0 & 1_{n-1} & 0 \\
0 & 1 & 0 & 0
\end{pmatrix};
\]
so that by putting $C=C_1^{-1}\cup s_1 C_1^{-1}$ we have $G(\BA)=G(k)\cdot T(\BA)' \cdot C$.

In what follows we will identify $(\BA^{\times})^n$ with $T(\BA)$ by means of the isomorphism $t\mapsto m(\lambda_t)$, and also with its image in $\wt{G(\BA)}$ by means of the isomorphism $t\mapsto (m(\lambda_t), 1)$.  We can then write $\wt{G(\BA)}=G(k)\cdot T(\BA)'\cdot \pi^{-1}(C)$, where $\pi: \wt{G(\BA)} \ra G(\BA)$ is the canonical projection. This shows the following result, which is an analogue of Lem. 20 on p. 67 of \cite[n. 47]{Weil 1965}.

\begin{lem}\label{Lem 20}
Let $\hat{E}$ be a tempered measure on $X(\BA)$, invariant under $G(k)$; let $T(\BA)''$ be a subset of
$T(\BA)'$ such that $T(\BA)'\subset T(k)\cdot T(\BA)'' \cdot C'$, where $C'$ is a compact subset of $T(\BA)$. Then,
for the function $g\mapsto \hat{E}(\omega(g)\Phi)$ to be bounded on $\wt{G(\BA)}$, uniformly in $\Phi$ on every compact
subset of $\CS(X(\BA))$, it is necessary and sufficient that it be so on $T(\BA)''$.
\end{lem}

For each $b\in \Her_n(k)$, denote by $b_1, \ldots, b_n$ the columns of the matrix $b$, and
write $b=(b_1, \ldots, b_n)$. We thus have $b_{\alpha}\in M_{n\times 1}(D)$ for $1\leq \alpha \leq n$. If $b=i_X(x)$ for some $x\in X$, then $b_{\alpha}=x^* \cdot Q \cdot x_{\alpha}$. For $0\leq \alpha \leq n$, we denote by
$\Her_n^{(\alpha)}(k)$ the set of elements $b=(b_1,\ldots, b_n)$ of $\Her_n(k)$ such that $b_1=\ldots=b_{\alpha}=0$
and $b_{\alpha+1}\neq 0$; $\Her_n(k)$ is thus the disjoint union of $\Her_n^{(\alpha)}(k)$ for $0\leq \alpha \leq n$.

We have the following analogue of Lem. 21 on p. 68 of \cite[n. 48]{Weil 1965}.

\begin{lem}\label{Lem 21}
Let $\hat{E}$ be a positive tempered measure on $X(\BA)$, invariant under $T(k)$, whose support is contained
in the union of $i_X^{-1}(\{b\})$ for $b\in \Her_n^{(0)}(k)$. Then the function $g\mapsto \hat{E}(\omega(g)\Phi)$ is bounded
on $T(\BA)'$, uniformly in $\Phi$ on every compact subset of $\CS(X(\BA))$.
\end{lem}
\begin{proof}
As before, we denote by $\Theta(T)$ the set of elements of $T(\BA)$ of the form $(a_{\tau_1}, \ldots, a_{\tau_n})$, with $\tau_{\alpha} \in \BZ$ for $1\leq \alpha \leq n$; put $\Theta'=\Theta(T)\cap T(\BA)'$; $\Theta'$ is the set of elements of the above form for which $\tau_1\leq \ldots \leq \tau_n\leq 0$. There is a compact subset $C'$ of $T(\BA)'$ such that $T(\BA)'=T(k) \cdot \Theta'\cdot C'$. Let $C_0$ be a compact subset of $\CS(X(\BA))$; let $C_0'$ be the set of $\omega(\theta)\Phi$ for $\theta\in C'$, $\Phi \in C_0$. Applying Lemma \ref{Lem 6} to the spaces
$X=M_{m\times n}(D)$,
$X^{(\alpha)}=M_{m\times 1}(D)$
for $1\leq \alpha \leq n$, $Y=M_{n\times 1}(D)$, and to the morphism $x\mapsto p(x)=x^*\cdot Q\cdot x_1$ of $X$ into $Y$, we conclude that there exists $\Phi_0 \in \CS(X(\BA))$ such that $|\omega(\theta)\Phi|\leq \Phi_0$ on the support of $\hat{E}$ for all $\theta \in \Theta'$, $\Phi \in C'_0$. The conclusion of the lemma follows.
\end{proof}

Let $b \in \Her_n(k)$; let $j$ be the canonical injection of $U(b)$ into $X$; $j$ then determines an injective mapping $j_{\BA}$ of $U(b)_{\BA}$ into $X(\BA)$, and
more precisely into $i_X^{-1}(\{b\})$.
Following the definition on p. 69 of \cite[n. 49]{Weil 1965}, we say a measure on $X(\BA)$ is {\sl supported by} ({\sl port\'{e}e par}) $U(b)_{\BA}$ if it is the
image under $j_{\BA}$ of a measure on $U(b)_{\BA}$. For example, this is so for the measure $\mu_b$, which by definition
is the image of $|\theta_b|_{\BA}$ under $j_{\BA}$, and which appears in Theorem \ref{Thm 2}.
When $b$ is a non-degenerate element of $\Her_n(k)$, it results from the remarks on p. 38 of \cite[n. 25]{Weil 1965}
that $j_{\BA}$ is an isomorphism of $U(b)_{\BA}$ onto $i_X^{-1}(\{b\})$; in this case, every measure of support
contained in $i_X^{-1}(\{b\})$ is supported by $U(b)_{\BA}$.

On the other hand, in the following, a place $v$ of $k$ will be given once and for all, and we write
$X(\BA)=X_v\times X'$. For $x\in X(\BA)$, we write $x=(x_v, x')$, where $x_v$ and $x'$ are the projections of $x$ onto $X_v$
and onto $X'$ respectively. We write similarly $U(b)_{\BA}=U(b)_v \times U(b)'$.

We have the following analogue of Lem. 22 on p. 70 of \cite[n. 49]{Weil 1965}.

\begin{lem}\label{Lem 22}
Let $b\in \Her_n(k)$, and let $H_v'$ be a subgroup of $H_v$ which acts transitively on $U(b)_v$.
Let $\mu$ be a positive tempered measure supported by $U(b)_{\BA}$  and invariant under $H_v'$. Then,
to every function $\Phi' \in \CS(X')$, there corresponds a constant $c(\Phi')$ such that for any $\Phi_v\in \CS(X_v)$ we have:
\begin{equation}
\int \Phi_v(x_v)\Phi'(x')\,\dif \mu((x_v, x'))=c(\Phi')\int_{U(b)_v}\Phi_v \cdot |\theta_b|_v.  \label{38}
\end{equation}
\end{lem}
\begin{proof}
By hypothesis, $\mu$ is the image of a measure $\nu$ on $U(b)_{\BA}$, i.e. the first member of (\ref{38}) is the integral of $\Phi_v(x_v)\Phi'(x')$ on $U(b)_v\times U(b)'$ with respect to $\nu$. Assume first $\Phi'\geq 0$. By hypothesis, the integral in question is finite whenever $\Phi_v$ is $\geq 0$ and belongs to $\CS(X_v)$, thus also whenever $\Phi_v$ is continuous and of compact support on $X_v$, and especially whenever $\Phi_v$ is continuous and of compact support on $U(b)_v$. It can thus be written as $\int \Phi_v \, \dif \nu_v$, where $\nu_v$ is a positive measure on $U(b)_v$. By the hypothesis made on $\mu$, $\nu_v$ is invariant under $H_v'$. We can then assume that $H_v'$ is closed in $H_v$ (if not, we replace it with its closure), and thus identify $U(b)_v$ with the homogeneous space determined by $H_v'$ and the stabilizer of one of its points in $H_v'$. But the definition of the gauge form $\theta_b$ in Theorem \ref{Thm 2} shows that it is invariant under $H$, up to a factor $\pm 1$; consequently, the measure $|\theta_b|_v$ is invariant under $H_v$, and especially under $H_v'$. The theorems on the uniqueness of  the invariant measure on homogeneous spaces (see \cite{Bourbaki 1963}, Chap. VII, $\S 2$, n. 6) then show that $\nu_v$ only differs from $|\theta_b|_v$ by a scalar factor $c(\Phi')$. The general case can be reduced to the special case $\Phi'\geq 0$ and thus follows immediately.
\end{proof}

Now we can prove the following analogue of Lem. 23 on p. 70 of \cite[n. 49]{Weil 1965}.

\begin{lem}\label{Lem 23}
Let $\hat{E}$ be a positive tempered measure, invariant under $T(k)$, which is the sum of measures $\hat{\mu}_b$
respectively supported by $U(b)_{\BA}$ for $b\in \Her_n(k)$. Assume that there exists a place $v$ of $k$ and a subgroup $H_v'$ of $H_v$ acting transitively on $U(b)_v$ for any $b\in \Her_n(k)$, such that
$\hat{E}$ is invariant under $H_v'$. Then the function $g\mapsto \hat{E}(\omega(g)\Phi)$ is bounded on $T(\BA)'$,
uniformly in $\Phi$ on every compact subset of $\CS(X(\BA))$.
\end{lem}
\begin{proof}
Let $\hat{E}_{\alpha}$, for $0\leq \alpha \leq n$, be the sum of $\hat{\mu}_b$ for $b\in \Her_n^{(\alpha)}(k)$; we will have $\hat{E}=\hat{E}_0+\ldots+\hat{E}_n$. If $t\in T(k)$, then $\bar{\lambda}_t$ determines a permutation on each of the sets $\Her_n^{(\alpha)}(k)$, so that each of the measures $\hat{E}_{\alpha}$ is invariant under $T(k)$. On the other hand, $H(\BA)$ leaves invariant each of the sets $i_X^{-1}(\{b\})$; with the hypotheses of the statement, it follows that $H_v'$ leaves invariant each of the measures $\hat{\mu}_b$, thus also each of the $\hat{E}_{\alpha}$; this satisfies thus the same hypotheses as $\hat{E}$, and we are reduced to dealing with $\hat{E}_{\alpha}$.

Thus let $\alpha$ be such that $0\leq \alpha \leq n$. There is a constant $q_1$, equal to 1 if $v=v_0$ and to $q_v$ if $v\neq v_0$, such that there exists, for every $\tau \in \mathbb{Z}$, an element $y$ of $k_v$ satisfying $q^{-\tau} \leq |y|_v \leq q_1 q^{-\tau}$; and there exists a compact subset $C$ of $\BA^{\times}$ such that every element $t$ of $\BA^{\times}$ satisfying $1\geq |t|_{\BA} \geq q_1^{-1}$ can be written as the form $\rho c$ with $\rho \in k$, $c\in C$; we denote by $C^n$ the compact subset of $T(\BA)$ formed of the elements $(c_1, \ldots, c_n)$ with $c_{\beta}\in C$ for $1\leq \beta \leq n$. Let $C_0$ be a compact subset of $\CS(X(\BA))$, and we apply Lemma \ref{Lem 6} to the space $X=M_{m\times n}(D)$ considered as a product of the spaces
\[
X^{(1)}=M_{m\times (\alpha+1)}(D), \quad X^{(2)}=\ldots=X^{(n-\alpha)}=M_{m\times 1}(D)
\]
in such a way that the projections of $x=(x_1, \ldots, x_n)$ on these spaces are respectively $(x_1, \ldots, x_{\alpha+1})$, $x_{\alpha+2}$, $\ldots$, $x_n$; we take $Y=M_{n\times (\alpha+1)}(D)$, and $p$ the morphism of $X$ into $Y$ given by
\[
p(x)=x^* \cdot Q\cdot (x_1, \ldots, x_{\alpha+1}).
\]
It is concluded that there exists $\Phi_0\in \CS(X(\BA))$, $\Phi_0\geq 0$, such that
\[
|\omega(m(\lambda_{\theta}))\omega(m(\lambda_c))\Phi(x)|\leq \Phi_0(x)
\]
for all $x\in X(\BA)$, $i_X(x)\in \Her_n^{(\alpha)}(k)$, $c\in C^n$, $\Phi \in C_0$, and $\theta$ belonging to the set $\Theta'_{\alpha}$ of elements $(a_{\tau_1}, \ldots, a_{\tau_n})$ of $\Theta(T)$ which satisfy the condition
\[
\tau_1=\ldots=\tau_{\alpha+1}\leq \ldots \leq \tau_n\leq 0.
\]
Furthermore, we can assume that $\Phi_0$ is of the form $\Phi_v(x_v)\Phi'(x')$, with
\[
\Phi_v \in \CS(X_v), \quad \Phi'\in \CS(X').
\]
Now let $t=(t_1, \ldots, t_n)$ be an element of $T(\BA)'$. For $1\leq \beta \leq \alpha$, let $y_{\beta}\in k_v$ be such that $|y_{\beta}|_v$ is between $|t_{\beta} t_{\alpha+1}^{-1}|_{\BA}$ and $q_1 |t_{\beta}t_{\alpha+1}|_{\BA}$; let $y_{\beta}=1$ for $\beta \geq \alpha+1$; we will have $|y_{\beta}|_v\geq 1$ for $1\leq \beta \leq n$. On the other hand, for $\beta \geq \alpha+1$, let $\tau_{\beta}\in \mathbb{Z}$ be such that $|a_{\tau_{\beta}}|_{\BA}=|t_{\beta}|_{\BA}$, and let $\tau_{\beta}=\tau_{\alpha+1}$ for $1\leq \beta \leq \alpha$. For every $\beta$, we will have
\[
1\geq |t_{\beta} y_{\beta}^{-1} a_{\tau_{\beta}}^{-1}|_{\BA} \geq q_1^{-1},
\]
so that we can write $t_{\beta}=\rho_{\beta} y_{\beta} a_{\tau_{\beta}} c_{\beta}$, with $\rho_{\beta}\in k$, $c_{\beta}\in C$, for $1\leq \beta \leq n$. By putting
\[
y=(y_1, \ldots, y_n), \quad \theta=(a_{\tau_1}, \ldots, a_{\tau_n}),
\]
we will thus have $t=\rho y \theta c$ with $\rho \in T(k)$, $y\in T_v$, $\theta \in \Theta'_{\alpha}$ and $c\in C^n$. As $\hat{E}_{\alpha}$ is invariant under $T(k)$, it follows that we have
\[
|\hat{E}_{\alpha}(\omega(m(\lambda_t))\Phi)|\leq \hat{E}_{\alpha}(\omega(m(\lambda_y))\Phi_0)
\]
for all $\Phi \in C_0$, $\Phi_0$ being chosen as above.

To evaluate the second member of this integral, we apply Lemma \ref{Lem 22} to each of the measures $\hat{\mu}_b$ for $b\in \Her_n^{(\alpha)}(k)$; denoting by $c_b(\Phi')$ the constant which appears in that lemma when we substitute $\hat{\mu}_b$ with $\mu$, we obtain
\[
\hat{E}_{\alpha}(\omega(m(\lambda_y))\Phi_0)=\sum_{b\in \Her_n^{(\alpha)}(k)}c_b(\Phi')\int_{U(b)_v} \omega(m(\lambda_y))\Phi_v \cdot |\theta_b|_v.
\]
As we have $y_{\beta}=1$ for $\beta \geq \alpha+1$, it results from the definition of $\Her_n^{(\alpha)}(k)$ that the automorphism $\bar{\lambda}_y$ of $\Her_n(k_v)$ determined by $\lambda_y$ leaves invariant all the elements of $\Her_n^{(\alpha)}(k)$. Note that we have
\[
|\hat{E}_{\alpha}(\omega(m(\lambda_y))\Phi_0)|=|y_1 \ldots y_{\alpha}|_v^{(-m+2n+4\epsilon-2)\delta/2} \hat{E}_{\alpha}(\Phi_0).
\]
As we have assumed that $m>2n+4\epsilon-2$, the exponent of the second member is $<0$. As we have $|y_{\beta}|_v\geq 1$ for all $\beta$, we obtain
\[
|\hat{E}_{\alpha}(\omega(m(\lambda_t))\Phi_0)|\leq \hat{E}_{\alpha}(\Phi_0),
\]
this inequality being valid for all $t\in T(\BA)'$ and $\Phi \in C_0$. This completes the proof.

\end{proof}

Now we can prove the main result of this section, which is an analogue of Thm. 4 on p. 72 of \cite[n. 50]{Weil 1965}. Note that by Theorem \ref{Thm 3}, if $m>2n+4\epsilon-2$, then the Siegel Eisenstein series $E(\Phi)$ is absolutely convergent for any $\Phi \in \CS(X(\BA))$ and this gives a positive tempered measure $E$ on $X(\BA)$.

\begin{thm}\label{Thm 4}
Assume that $m>2n+4\epsilon-2$.
Let $v$ be a place of $k$ such that $U(0)_v$ is not empty,
and $H_v'$ a subgroup of $H_v$ acting transitively on $U(b)_v$ for any $b\in \Her_n(k)$.
Let $E'$ be a positive tempered measure on $X(\BA)$,
invariant under $G(k)$ and under $H_v'$,
and such that $E'-E$ is a sum of measures supported by $U(b)_{\BA}$ for $b\in \Her_n(k)$. Then we have $E'=E$.
\end{thm}
\begin{proof}
With the notations of Theorems \ref{Thm 2} and \ref{Thm 3}, we have $E=\sum_{0\leq r\leq n} E_{X_r}$; $E_X$ is the sum of the measures $|\theta_b|_{\BA}$ respectively supported by $U(b)_{\BA}$, while $E_{X_r}$ has its support contained in $X_r(\BA)$ for any $r<n$. On the universal domain, let $U$ be the set of points of $X$ of maximal rank; it is $k$-open; it is an orbit for the group $\mathrm{Aut}(V)$; for any $b\in \Her_n(k)$, $U(b)$ is a subvariety of $U$ and is thus $k$-closed in $U$. Let $F=X-U$; it is a $k$-closed subset of $X$, invariant under the group $\mathrm{Aut}(V)$ and especially under $H\subset U(V)$, which contains $X_r$ whenever $r<n$. Consequently, $F(\BA)$ is a closed subset of $X(\BA)$, invariant under $H(\BA)$ and obviously also under $\mathrm{Aut}(X_k)$, which contains $X_r(\BA)$ for $r<n$ and has no common point with $U(b)_{\BA}$ for any $b\in \Her_n(k)$. It follows that $E_X$ is the restriction of $E$ to the open set $X(\BA)-F(\BA)$, and that the sum $\sum E_{X_r}$ over $0\leq r <n$ is the restriction of $E$ to $F(\BA)$. The hypothesis made on $E'$ then implies that the restriction $\hat{E}$ of $E'$ to $X(\BA)-F(\BA)$ is the sum of the measures $\hat{\mu}_b$ respectively supported by the $U(b)_{\BA}$ for $b\in \Her_n(k)$, and that the restriction of $E'$ to $F(\BA)$ is the same as that of $E$, so that we have $E'-\hat{E}=E-E_X$; furthermore, as $E'$ and $F(\BA)$ are invariant under $H_v'$ and under $\mathrm{Aut}(X_k)$, it is the same for $\hat{E}$, which thus satisfies the hypotheses of Lemma \ref{Lem 23}.

According to that lemma, the function $g\mapsto \hat{E}(\omega(g)\Phi)$ is bounded on $T(\BA)'$, uniformly in $\Phi$ on every compact subset of $\CS(X(\BA))$. This conclusion can be applied in particular to $E_X$, which is deduced from $E$ as $\hat{E}$ is from $E'$; it can be applied also to the tempered measure $E''$ given by
\[
E''=E'-E=\hat{E}-E_X.
\]
But this measure is invariant under $G(k)$, since $E$ and $E'$ are so; we can thus apply Lemma \ref{Lem 20}, which shows that the function $g\mapsto  E''(\omega(g)\Phi)$ is bounded on $\wt{G(\BA)}$, for any $\Phi \in \CS(X(\BA))$. For every $\Phi \in \CS(X(\BA))$, we denote by $M(\Phi)$ the supremum of $|E''(\omega(g)\Phi)|$ for $g \in \wt{G(\BA)}$; we have $M(\omega(g)\Phi)=M(\Phi)$ for all $g\in \wt{G(\BA)}$.

The measure $E''$ is the sum of the measures $\mu''_b=\hat{\mu}_b-\mu_b$, where $\mu_b$ denotes once again the measure $|\theta_b|_{\BA}$ supported by $U(b)_{\BA}$. We thus have, for $\Phi \in \CS(X(\BA))$:
\[
E''(\Phi)=\sum_{b\in \Her_n(k)}\int \Phi \,\dif \mu''_b;
\]
in this formula, the series of the second member is absolutely convergent, uniformly in $\Phi$ on every compact subset of $\CS(X(\BA))$, since it is obviously also the series similarly formed by means of the positive measures $\hat{\mu}_b$ and $\mu_b$. Let $b^*\in \Her_n(\BA)$, we then have, for $\Phi \in \CS(X(\BA))$:
\[
\omega(n(b^*))\Phi(x)=\Phi(x)\psi(q_{b^*}(x))),
\]
and consequently
\[
E''(\omega(n(b^*))\Phi)=\sum_{b\in \Her_n(k)}\psi(\frac{\kappa}{2}\tau(bb^*))\int \Phi \,\dif \mu''_b.
\]
We can consider this formula as giving the expansion of the first member into Fourier series on the compact group $\Her_n(\BA)/\Her_n(k)$. As the first member, in absolute value, is $\leq M(\Phi)$, we have, by the Fourier formulas, that
\[
|\int \Phi \, \dif \mu''_b|\leq M(\Phi),
\]
and consequently, replacing $\Phi$ by $\omega(g)\Phi$,
\begin{equation}
|\int \omega(g)\Phi \cdot \dif \mu''_b|\leq M(\Phi), \label{39}
\end{equation}
this integral being valid for all $g\in \wt{G(\BA)}$, $b\in \Her_n(k)$ and $\Phi \in \CS(X(\BA))$.

Taking for $\Phi$ the form $\Phi_v(x_v)\Phi'(x')$, with $\Phi_v\in \CS(X_v)$, $\Phi'\in \CS(X')$. By the hypotheses made on $E'$, the measures $\hat{\mu}_b$ are invariant under $H_v'$, and the same is true for $\mu_b$; we can thus apply Lemma \ref{Lem 22} to them. Consequently, we can write
\[
\int \Phi \,\dif \mu''_b=c_b(\Phi')\int_{U(b)_v} \Phi_v\cdot |\theta_b|_v.
\]
We now replace $\Phi$ by $\omega(m(\lambda_t))\Phi$ with $t\in T_v$ in this formula; this is equivalent to not changing $\Phi'$ but replacing $\Phi_v$ by $\omega(m(\lambda_t))\Phi_v$, the later function being given by the formula analogous to (\ref{37}). If we put $b'=b \bar{\lambda}_t$, then this gives, by (\ref{36}):
\begin{equation}
\int \omega(m(\lambda_t))\Phi \cdot \dif \mu''_b=\chi(m(\lambda_t)) c_b(\Phi')|t_1 \ldots t_n|_v^{(-m+2n+4\epsilon-2)\delta/2}\int_{U(b')_v} \Phi_v \cdot |\theta_{b'}|_v. \label{40}
\end{equation}
Denote by $F(b')$ the integral which appears in the second member; then Prop. 6 on p. 54 of \cite[n. 37]{Weil 1965} shows that it is a continuous function of $b'\in \Her_n(k_v)$, so that $F(b')$ tends to $F(0)$ when all the $|t_{\alpha}|_v$ tend to 0. As the exponent of $|t_1 \ldots t_n|_v$ in the second member of (\ref{40}) is $<0$ by the assumption $m>2n+4\epsilon-2$, and the second member must remain bounded for all $t\in T_v$, we conclude that $c_b(\Phi')F(0)=0$. But $F(0)$ is given by
\[
F(0)=\int_{U(0)_v}\Phi_v \cdot |\theta_0|_v,
\]
and, by hypothesis, $U(0)_v$ is not empty; we can thus choose $\Phi_v$ in such a way that $F(0)$ is not zero. We thus have $c_b(\Phi')=0$, and consequently $\int \Phi~ \dif \mu''=0$ whenever $\Phi$ is of the form $\Phi_v(x_v)\Phi'(x')$. This implies obviously $\mu''_b=0$. As this is so for any $b\in \Her_n(k_v)$, we thus have $E''=0$, i.e. $E'=E$.

\end{proof}
Observe that Theorem \ref{Thm 4} provides a characterization of the measure $E_X$, by induction
on the rank $n$ of $X$ over $M_m(D)$, from $E_{0}=\delta_0$.

\section{The theta integral and the Siegel-Weil formula}\label{Section 6}
In this section, we will first study the theta integrals and give a convergence criterion, and then prove the Siegel-Weil formula, which is an equality relating the Siegel Eisenstein series with the theta integral.

For $\Phi \in \CS(X(\BA))$, define the associated {\sl theta integral} by
\[
I(\Phi)=\int_{H(\BA)/H(k)}\sum_{\xi \in X(k)}\Phi(h\xi)\cdot \dif h,
\]
where $\dif h$ is the Haar measure on $H(\BA)$ such that $\vol(H(\BA)/H(k))=1$. Note that since we are integrating on the quotient group $H(\BA)/H(k)$ of left cosets, we take the action $\Phi(h\xi)=\omega(h^{-1})\Phi(\xi)$ here.

We have the following convergence criterion for the theta integral $I(\Phi)$, which is an analogue of Prop. 8 on p. 75 of \cite[n. 51]{Weil 1965}. Recall $r$ is the Witt index of the $\eta$-hermitian space $V$.

\begin{prop}\label{Prop 8}
The theta integral $I(\Phi)$ is absolutely convergent for any $\Phi \in \CS(X(\BA))$ whenever $r=0$ or $m-r>n+2\epsilon-1$.
\end{prop}
\begin{proof}
If $r=0$, then $H$ is $k$-anisotropic and thus $H(\BA)/H(k)$ is compact (see for example Thm 5.1.1 on pp. 582--583 of \cite{Conrad 2012}), whence the theta integral is absolutely convergent.

Now assume $r>0$. Then we can choose a basis of $V$ for which the $\eta$-hermitian form on $V$ is given by a matrix of the form
\[
Q=\begin{pmatrix} 0 & 0 & 1_r \\ 0 & Q_0 & 0 \\ \eta\cdot 1_r & 0 & 0  \end{pmatrix},
\]
where $Q_0$ is the matrix (of order $m-2r$) of an anisotropic $\eta$-hermitian form.
Let $T\cong (\BG_m)^r$ be the maximal split torus in $H$ consisting of diagonal matrices of order $m$ whose diagonal elements are
\[
(t_1, \ldots, t_r, 1, \ldots, 1, t_1^{-1}, \ldots, t_r^{-1}).
\]
with each $t_i\in \BG_m$. Let $P_0$ be a minimal parabolic subgroup of $H^0$ which contains $T$.

Let $\rho: H \ra \Aut(X)$ be the representation given by $\rho(h)x=hx$. For each character $\lambda$ of $T$, let $m_{\lambda}$ be the multiplicity of $\lambda$. The weights of $\rho$ are $x_i$ and $-x_i$ for $1\leq i \leq r$, each with multiplicity $\delta n$.

By Lemma \ref{Lem 5}, it suffices to show that
\[
\int_{\Theta^{+}} \prod_{\lambda} \sup(1, |\lambda(\theta)|_{\BA}^{-1})^{m_{\lambda}}\cdot |\Delta_{P_0}(\theta)|_{\BA}^{-1} \,\dif \theta
\]
is convergent whenever $m-r>n+2\epsilon-1$.

Note that
\[
\Theta^{+}:=\Theta(0)=\{(a_{\tau_1}, \ldots, a_{\tau_r}): 0 \leq \tau_r \leq \ldots \leq \tau_1\}.
\]
For $\theta=(a_{\tau_1}, \ldots, a_{\tau_r})\in \Theta^{+}$, we have, by Lemma \ref{Lem 5.2}, that
\[
\Delta_{P_0}(\theta)^{-1}=\prod_{1\leq i \leq r} a_{\tau_i}^{\delta(m-2i+2-2\epsilon)},
\]
and hence
\[
|\Delta_{P_0}(\theta)|_{\BA}^{-1}=\prod_{1\leq i \leq r} q^{-\delta\tau_i(m-2i+2-2\epsilon)}.
\]

Thus we have
\[\begin{aligned}
&\quad \int_{\Theta^{+}}\prod_{\lambda} \sup(1, |\lambda(\theta)|_{\BA}^{-1})^{m_{\lambda}} \cdot |\Delta_{P_0}(\theta)|_{\BA}^{-1}\,\dif \theta\\
&=\sum_{0 \leq \tau_r \leq \ldots \leq \tau_1} \prod_{1\leq i \leq r} q^{-\delta \tau_i(m-n-2i+2-2\epsilon)}\\
&=c_1 \ldots c_{r-1}\sum_{\tau_r\geq 0} q^{-r\delta \tau_r(m-n-r+1-2\epsilon)},
\end{aligned}\]
where $c_j=(1-q^{-\delta \tau_j(m-n-j+1-2\epsilon)})^{-1}$.

Note that the above multiple series converges if and only if $m-n-r+1-2\epsilon>0$, i.e. $m-r>n+2\epsilon-1$.
The desired result follows.
\end{proof}

\begin{lem} \label{Lem 5.2}
For $\theta=(a_{\tau_1}, \ldots, a_{\tau_r})\in \Theta^{+}$, we have
\[
\Delta_{P_0}(\theta)^{-1}=\prod_{1\leq i \leq r} a_{\tau_i}^{\delta(m-2i+2-2\epsilon)}.
\]
\end{lem}
\begin{proof}
The restriction of $\Delta_{P_0}^{-1}$ to $T$ is the product of all the positive roots (see \cite[p. 17]{Weil 1965}). Now the positive roots are the following: $x_i-x_j$ and $x_i+x_j$ for $1\leq i <j \leq r$, each with multiplicity $\delta$; $x_i$ and $2x_i$ for $1\leq i \leq r$, with multiplicities $\delta(m-2r)$ and $\delta(1-\epsilon)$ respectively (see \cite[p. 76]{Weil 1965}). So the sum of all the positive roots are
\[
2\delta x_1(r-1)+2\delta x_2(r-2)+\ldots+2\delta x_{r-1}+\sum_{i=1}^r \delta(m-2r+2-2\epsilon)x_i=\sum_{i=1}^r \delta(m-2i+2-2\epsilon)x_i.
\]
\end{proof}

Now we can show the Siegel-Weil formula. We follow the proof of Thm. 5 on p. 76 of \cite[n. 52]{Weil 1965}.

\begin{thm}\label{Thm 5}
Assume that $m>2n+4\epsilon-2$.
Let $v$ be a place of $k$ such that $U(0)_v$ is not empty,
and $H_v'$ a subgroup of $H_v$ which acts transitively on $U(b)_v$ for any $b\in \Her_n(k)$.
 Let $\nu$ be a positive measure on $H(\BA)/H(k)$, invariant under $H_v'$, such that $\nu(H(\BA)/H(k))=1$
 and that the integral
\[
I_{\nu}(\Phi)=\int_{H(\BA)/H(k)}\sum_{\xi \in X(k)}\Phi(h\xi)\cdot d\nu(h)
\]
is absolutely convergent for any $\Phi \in \CS(X_A)$.
Then we have
\[
I_{\nu}(\Phi)=E(\Phi),
\]
and for every $b\in \Her_n(k)$ we have
\begin{equation}
\int_{H(\BA)/H(k)}\sum_{\xi \in U(b)_k}\Phi(h \xi)\cdot \dif \nu(h)=\int \Phi \, \dif \mu_b,  \label{41}
\end{equation}
where $\mu_b$ is the measure $|\theta_b|_{\BA}$ determined on $U(b)_{\BA}$ by the gauge form $\theta_b$
defined in Theorem \ref{Thm 2}.
\end{thm}
\begin{proof}
We proceed by induction on the rank $n$ of $X$ over $\CA$, where $\CA=M_m(D)$ is equipped with an involution given by $x\mapsto Q^{-1} \cdot x^* \cdot Q$, here $Q$ is the invertible $\eta$-hermitian matrix over $D$ of order $m$ which is used to define the $\eta$-hermitian form on $V$. We fist show that the hypotheses made on
$v$ and $\nu$ with regard to $X$ imply that $v$ and $\nu$ have similar properties with regard to each $\CA$-module $X'$ of rank $n'\leq n$. Since all the $\CA$-modules of the same rank are isomorphic,
it suffices to prove this assertion when $X'$ is a submodule of $X$. Concerning the condition imposed
on $\nu$, the assertion is obvious. Concerning the place $v$, let $X_v=X\otimes_k k_v$, $X'_v=X'\otimes_k k_v$, and $\CA_v=\CA\otimes_k k_v$. Then $X_v$ and $X'_v$ are $\CA_v$-modules of rank denoted respectively by $n_v$ and $n'_v$. For every $b'\in \Her_{n'}(k)$, we denote by $U'(b')_v$ the set of elements $x'$ of $X_v'$, of maximal rank in $X_v'$, which satisfy $i_{X'}(x')=b'$. Consider first the hypothesis $U(0)_v\neq \emptyset$, which we want to show that it implies $U'(0)_v\neq \emptyset$.
Assume first that $\CA_v$ is of type (I), i.e. $\CA_v$ is of the form $M_{m_v}(\mathfrak{K})$, where $m_v$ is a positive integer and $\mathfrak{K}$ is a division algebra over $k_v$; the involution on $\CA_v$ is then defined by an involution on $\mathfrak{K}$, and by a matrix $Q_v \in M_{m_v}(\mathfrak{K})$ which is $\eta_v$-hermitian with respect to this involution. Saying  that $U(0)_v \neq \emptyset$ then amounts to saying that $Q_v$ is of Witt index $\geq n_v$, and $U'(0)_v\neq \emptyset$ amounts similarly to saying that this index is $\geq n_v'$; for $n'\leq n$, the first assertion implies obviously the second.
If $\CA_v$ is of type (II), i.e. it is of the form $\CA_v=M_{m_v}(\mathfrak{K})\oplus M_{m_v}(\mathfrak{K}')$, where $\mathfrak{K}$ and $\mathfrak{K}'$ are two division algebras over $k_v$ which are anti-isomorphic, then $U(0)_v \neq \emptyset$ and $U'(0)_v\neq \emptyset$ are respectively equivalent to $m_v\geq 2n_v$ and $m_v\geq 2n_v'$ by \cite[n. 23]{Weil 1965}, and we draw the same conclusion; besides we note that, in this case, $\epsilon=\frac{1}{2}$ by \cite[n. 26]{Weil 1965}, thus $m>2n$ by the assumption $m>2n+4\epsilon-2$, so that certainly $U(0)_v \neq \emptyset$ and $U'(0)_v\neq \emptyset$. Next we consider the transitivity of $H_v'$ on the sets $U(b)_v$, $U'(b')_v$. Recall $X=M_{m\times n}(D)$, and identify $X'$ with the submodule of $X$ of elements of the form $(x_1, \ldots, x_{n'}, 0, \ldots, 0)$; denote by $X''$ the submodule of $X$ of elements of the form $(0, \ldots, 0, x_{n'+1}, \ldots, x_n)$, so that $X=X'\oplus X''$. Let $b'\in \Her_{n'}(k)$, and let $b$ be the element of $\Her_n(k)$ given by the matrix $\begin{pmatrix} b' & 0\\ 0 & 0 \end{pmatrix}$ with $n$ lines and $n$ columns. Assume first that $\CA_v$ is of type (I); with the same notations as above, we can identify $\Her_n(k_v)$ with the space of $\eta_v$-hermitian matrices of $n_v$ lines and $n_v$ columns  over $\mathfrak{K}$, and do the same for $\Her_{n'}(k_v)$. The canonical isomorphism of $\Her_n(k)\otimes k_v$ onto $\Her_n(k_v)$ induces a $k$-linear mapping of $\Her_n(k)$ into $\Her_n(k_v)$. Let $b_v \in \Her_n(k_v)$ be the image of $b$ under this mapping; if $b_v'$ is the element of $\Her_{n'}(k_v)$ deduced similarly from $b'$, then $b_v=\begin{pmatrix} b_v' & 0\\ 0 & 0 \end{pmatrix}$. By \cite[n. 19]{Weil 1965}, $U(b)_v$ can be identified with the set of matrices $x$ of $m_v$ lines and $n_v$ columns over $\mathfrak{K}$, of maximal rank (i.e. equal to $n_v$), which satisfy $Q_v[x]:=x^*\cdot Q_v \cdot x=b_v$; we have a similar assertion for $U'(b')_v$. By hypothesis, $U(0)_v$ is not empty, which means that $Q_v$ is of Witt index $\geq n_v$; we deduce easily that $U(b)_v$ is not empty; and we choose $a\in U(b)_v$. Then, if $a'$, $a''$ are the projections of $a$ on $X_v'$, $X_v''$ for the decomposition $X_v=X_v'\oplus X_v''$, we have
\[
Q_v[a]=Q_v[(a', a'')]=b_v=\begin{pmatrix} b_v' & 0\\ 0 & 0 \end{pmatrix},
\]
thus $Q_v[a']=b_v'$; moreover, as $a$ is of maximal rank (equal to $n_v$) in $X_v$, $a'$ must be of maximal rank (equal to $n_v'$) in $X_v'$; thus $a'\in U'(b')_v$. Let then $x'\in U'(b')_v$. By Prop. 3 of \cite[n. 22]{Weil 1965}, there exists $h\in H_v$ such that $x'=ha'$, Then $a$ and $ha$ both belong to $U(b)_v$, so that, by hypothesis, there exists $h'\in H_v'$ such that $ha=h'a$, whence $x'=h'a'$. This shows that $H_v'$ acts transitively on $U'(b')_v$. The proof is similar when $\CA_v$ is of type (II).

Let then $v$ and $\nu$ satisfy the hypotheses of Theorem \ref{Thm 5} with regard to the module $X$.
For $n=0$, the assertion of the theorem is reduced to $I_{\nu}=\delta_0$, which is an obvious consequence of the hypothesis $\nu(H(\BA)/H(k))=1$, here $\delta_0$ is the measure on $X(\BA)$ given by $\delta_0(\Phi)=\Phi(0)$.
We proceed by induction on $n$, and suppose $n\geq 1$. Since by hypothesis $I_{\nu}(\Phi)$ is convergent for any $\Phi \in \CS(X(\BA))$, Lem. 2 on p. 5 of \cite[n. 2]{Weil 1965} joint with Lem. 5 on p. 194 of \cite[n. 41]{Weil 1964} show immediately that $I_{\nu}$ is a positive tempered measure. Now Thm. 6 on p. 193 of \cite[n. 41]{Weil 1964} and Prop. 9 on p. 210 of \cite[n. 51]{Weil 1964} show that $I_{\nu}$ is invariant under $G(k)$; it is also obviously invariant under $H_v$ for a place $v$ of $k$ such that $U(0)_v$ is non-empty.
Similarly, if we denote by $I_{\nu, b}(\Phi)$ the first member of (\ref{41}), then $I_{\nu, b}$ is a positive tempered measure. Let $I_{\nu, X}$ be the sum of $I_{\nu, b}$ for $b\in \Her_n(k)$; we can consider $I_{\nu, X}$ as defined by the integral similar to that which defines $I_{\nu}$, but where the summation is restricted to the elements $\xi$ of $X(k)$ which are of maximal rank in $X(k)$. Similarly, for the submodule $X_r$ of $X$, where $0\leq r \leq n-1$, denote by $I_{\nu, X_r}$ the positive tempered measure defined by the integral similar to that which defines $I_{\nu, X}$, but where the summation is restricted to the elements $\xi$ of $X_r(k)$ which are of maximal rank in $X_r(k)$. Taking into account of Theorem \ref{Thm 2}, we see that Theorem \ref{Thm 5} for $X$ implies that $I_{\nu, X}=E_X$; as a result, the induction hypothesis implies that $I_{\nu, X_r}=E_{X_r}$ for the submodule $X_r$ of $X$ whenever $r<n$. Thus we have, by this hypothesis:
\begin{equation}
I_{\nu}=\sum_{b\in \Her_n(k)} I_{\nu, b}+\sum_{0 \leq r \leq n-1} E_{X_r}. \label{42}
\end{equation}
According to Theorem \ref{Thm 3}, the second sum of the second member is just $E-E_X$. On the other hand, according to Prop. 3 on p. 34 of \cite[n. 22]{Weil 1965}, those of $U(b)_k$ which are non-empty are orbits of $H(k)$ in $X(k)$; then formula (11) on p. 15 of \cite[n. 7]{Weil 1965} shows that the measures $I_{\nu, b}$ are respectively supported by $U(b)_{\BA}$. Consequently, $I_{\nu}$ satisfies all the hypotheses of Theorem \ref{Thm 4}, thus $I_{\nu}=E$, and $I_{\nu, X}=E_X$ by (\ref{42}). As $I_{\nu, b}$ and $\mu_b$ are the restrictions of $I_{\nu, X}$ and of $E_X$ to the set $i_X^{-1}(\{b\})$ respectively, it follows that $I_{\nu, b}=\mu_b$ for any $b\in \Her_n(k)$.

This completes the proof of this theorem.
\end{proof}

Now we take for $\nu$ the Haar measure on $H(\BA)/H(k)$ normalized by $\vol(H(\BA)/H(k))=1$; then $I_{\nu}(\Phi)$ is just the theta integral $I(\Phi)$ defined before. We choose a place $v$ of $k$ such that $U(0)_v$ is not empty (this holds for almost every $v$), and take $H_v'=H_{v}$. Then this Haar measure and this place $v$ satisfy the hypotheses of Theorem \ref{Thm 5} (as can be seen from the proof of this theorem).
Note that if $m>2n+4\epsilon-2$, then this implies automatically that $r=0$ or $m-r>n+2\epsilon-1$ (since $r\leq \frac{m}{2}$), and thus $I(\Phi)$ is absolutely convergent, and $I(\Phi)=E(\Phi)$ by the above theorem, for any $\Phi \in \CS(X(\BA))$.

More generally, for $\Phi \in \CS(X(\BA))$ and $g\in \wt{G(\BA)}$, let
\[
I(g,\Phi)=I(\omega(g)\Phi).
\]
Then the theta integral $I(g,\Phi)$ is absolutely convergent whenever $r=0$ or $m-r>n+2\epsilon-1$.

\begin{cor}
Assume $m>2n+4\epsilon-2$. Then for all $\Phi \in \CS(X(\BA))$ and $g\in \wt{G(\BA)}$,

(i) $I(g, \Phi)$ is absolutely convergent, $E(g,s,\Phi)$ is holomorphic at $s=s_0$, where $s_0=\alpha(m-n+1-2\epsilon)/2$;

(ii) moreover, we have
\[
I(g,\Phi)=E(g,s_0,\Phi).
\]
\end{cor}

{\bf Remark 1}. Let $H_1=\{h\in H: \nu_K(h)=1\}$, where $\nu_K: M_m(D) \ra K$ is the reduced norm, and we recall that $K$ is the center of the division algebra $D$. For example, if $D=k$ and $V$ is a non-degenerate quadratic space over $k$, then $H=O(V)$ and $H_1=H^0=SO(V)$. When $\epsilon=\frac{3}{4}$ and $m\geq 2$, or $\epsilon=1$ and $m\geq 3$, we denote by $\tilde{H}$ the simply connected covering of $H_1$ (the spin group). Now we can define the integrals $I_1(\Phi)$ and $\tilde{I}(\Phi)$ by substituting $H_1$ and $\tilde{H}$ (when defined) for $H$ in the definition of the theta integral $I(\Phi)$ respectively, with similarly normalized Haar measures. As in \cite[n. 51--52]{Weil 1965}, we can show that $I_1(\Phi)$ and $\tilde{I}(\Phi)$ are absolutely convergent for any $\Phi \in \CS(X(\BA))$ whenever $r=0$ or $m-r>n+2\epsilon-1$, and we can show analogues of Theorem \ref{Thm 5} for $I_1(\Phi)$ and $\tilde{I}(\Phi)$.  We omit the details here. \\

{\bf Remark 2}. As in \cite[n. 53--56]{Weil 1965}, Theorem \ref{Thm 5} can be applied to study the Tamagawa numbers and the approximation theorems for groups of the form $H$, $H_1$ or $\tilde{H}$ defined above. We refer to \cite{Weil 1965} for more information.\\

{\bf Remark 3}. As in \cite{Ral84} and \cite{Li92}, we can derive a Rallis inner product formula over function fields from the Siegel-Weil formula established here and the basic identity of Piatetski-Shapiro and Rallis (\cite[p. 3]{PSR}). Since the Siegel-Weil formula established here only applies for $m>2n+4\epsilon-2$, the corresponding Rallis inner product formula is quite restrictive, and to establish a more general Rallis inner product formula will require more general Siegel-Weil formulas, as in the number field case (cf. \cite{Kudla-Rallis 1994}, \cite{Gan-Takeda 2011}, \cite{Yamana 2014}, \cite{Gan-Qiu-Takeda 2014}). The derivation of the Rallis inner product formula over function fields is similar to that over number fields, and we omit the details.

\end{document}